\documentclass[12pt]{amsart}
\usepackage[margin=1.2in]{geometry}
\usepackage{graphicx,latexsym}
\usepackage{comment}
\usepackage{tikz}
\usepackage{amsfonts, amssymb, amsmath, amsthm, bm, hyperref}
\usepackage{tikz}
\usetikzlibrary{matrix,arrows}
\allowdisplaybreaks

\newcommand{\N}{\mathbb{N}}
\newcommand{\Z}{\mathbb{Z}}
\newcommand{\R}{\mathbb{R}}
\newcommand{\C}{\mathbb{C}}
\newcommand{\Sp}{\overleftarrow{\mathcal{S}}}
\newcommand{\Si}{\overrightarrow{\mathcal{S}}}
\newcommand{\Sb}{\mathcal{S}}

\newcommand{\HH}{\mathbb{H}}

\newcommand{\dx}{{\rm d}x }

\newcommand{\dt}{{\rm d}t }

\newcommand{\csn}{\operatorname{csn}}

\newcommand{\beq}{\begin{eqnarray}}
\newcommand{\eeq}{\end{eqnarray}}

\newcommand{\beqs}{\begin{eqnarray*}}
\newcommand{\eeqs}{\end{eqnarray*}}

\newtheorem{theorem}{Theorem}[section]
\newtheorem{proposition}[theorem]{Proposition}
\newtheorem{lemma}[theorem]{Lemma}

\theoremstyle{definition}

\newtheorem{example}[theorem]{Example}

\theoremstyle{remark}
\newtheorem{remark}[theorem]{Remark}

\numberwithin{equation}{section}

\begin{document}
\title[The Stieltjes moment problem in Gelfand-Shilov spaces]{Solution to the Stieltjes moment problem in Gelfand-Shilov spaces}
\author[A. Debrouwere]{Andreas Debrouwere}
\address{Department of Mathematics: Analysis, Logic and Discrete Mathematics, Ghent University, Krijgslaan 281, 9000 Gent, Belgium}
\email{Andreas.Debrouwere@UGent.be}
\thanks{The author is supported by  FWO-Vlaanderen, via the postdoctoral grant 12T0519N}

\begin{abstract}
We characterize the surjectivity and the existence of a continuous linear right inverse of the Stieltjes moment mapping on Gelfand-Shilov spaces, both of Beurling and Roumieu type, in terms of their defining weight sequence. As a corollary, we obtain some new results about the Borel-Ritt problem in spaces of ultraholomorphic functions on the upper half-plane.
\end{abstract}

\maketitle
\section{Introduction}
In 1939, Boas \cite{Boas} and P\'olya \cite{Polya} independently showed  that for every sequence $(a_p)_{p \in \N}$ of complex numbers there is a function $F$ of bounded variation such that
$$
\int_0^\infty x^p {\rm d} F(x) =a_p,\qquad p \in \N.
$$
 A. J. Dur\'an \cite{Duran} (see also \cite{D-E}) improved this result in 1989 by showing that for every sequence  $(a_p)_{p \in \N}$ of complex numbers the infinite system of linear equations
\begin{equation}
\int_0^\infty x^p \varphi(x) \dx =a_p,\qquad p \in \N,
\label{moment-problem}
\end{equation}
admits a solution $\varphi \in \mathcal{S}(0,\infty)$ (= the space of rapidly decreasing smooth functions with support in $[0,\infty)$).  Over the past 20 years, various authors studied the \emph{(unrestricted) Stieltjes moment problem \eqref{moment-problem}} in the context of Gelfand-Shilov spaces \cite{G-S}; see \cite{C-K-Y, C-C-K,L-S08, L-S09, D-J-S}. In this article, we provide a complete solution to this problem. 

In order to be able to discuss our results, we need to introduce some notation; see Section \ref{sect-prel} for unexplained notions concerning weight sequences.
Let $(M_p)_{p \in \N}$ be a weight sequence. We define $\Sb_{(M_p)}(0,\infty)$ as the space consisting of all $\varphi \in \mathcal{S}(0,\infty)$ such that 
\begin{equation}
\sup_{p\in \N} \sup_{x \geq 0} \frac{h^px^p|\varphi^{(n)}(x)|}{M_p} < \infty
\label{norm}
\end{equation}
for all $h > 0$ and $n \in \N$. Similarly, we define $\Sb_{\{M_p\}}(0,\infty)$ as the space consisting of all  $\varphi \in \mathcal{S}(0,\infty)$ such that there is $h > 0$ for which \eqref{norm} holds for all $n \in \N$. $\Sb_{(M_p)}(0,\infty)$ and $\Sb_{\{M_p\}}(0,\infty)$ are endowed with their natural Fr\'echet space and $(LF)$-space topology, respectively. Next, we define $\Lambda_{(M_p)}$ and $\Lambda_{\{M_p\}}$ as the sequence spaces consisting of all $a = (a_p)_{p \in \N} \in \C^\N$ such that 
$$
\sup_{p \in \N} \frac{h^p|a_p|}{M_p} < \infty
$$ 
for all $h > 0$ and some $h > 0$, respectively. $\Lambda_{(M_p)}$ and $\Lambda_{\{M_p\}}$ are endowed with their natural Fr\'echet space and $(LB)$-space topology, respectively. If $(M_p)_{p \in \N}$  satisfies $(\operatorname{dc})$, the \emph{Stieltjes moment mapping}
$$
\mathcal{M}: \Sb_{\ast}(0,\infty) \rightarrow \Lambda_{\ast}: \varphi \rightarrow \left( \int_0^\infty x^p \varphi(x) \dx \right)_{p \in \N}
$$
is well-defined and continuous, where $\ast$ stands for either $(M_p)$ or $\{M_p\}$. Jim\'enez-Garrido, Sanz and the author \cite{D-J-S} characterized the surjectivity of  $\mathcal{M}: \Sb_{\{M_p\}}(0,\infty) \rightarrow \Lambda_{\{M_p\}}$ in the following way; see \cite{L-S09} for earlier work in this direction.

\begin{theorem}\label{prev-char}\cite[Thm.\ 3.5]{D-J-S}
Let $(M_p)_{p \in \N}$ be a weight sequence satisfying $\operatorname{(slc)}$ (= $(M_p/p!)_{p \in \N}$ satisfies $\operatorname{(lc)}$) and $\operatorname{(dc)}$. If the mapping $\mathcal{M}: \Sb_{\{M_p\}}(0,\infty) \rightarrow \Lambda_{\{M_p\}}$ is surjective, then $(M_p)_{p \in \N}$ satisfies 
\begin{itemize}
\item[$(\gamma_2)$] $\displaystyle \sup_{p \in \Z_+}\frac{(M_p/M_{p-1})^{1/2}}{p}\sum_{q = p}^\infty \frac{1}{(M_q/M_{q-1})^{1/2}}< \infty$.
\end{itemize}
If, in addition, $(M_p)_{p \in \N}$ satisfies $(\operatorname{mg})$, $(\gamma_2)$ implies that $\mathcal{M}: \Sb_{\{M_p\}}(0,\infty) \rightarrow \Lambda_{\{M_p\}}$ is surjective.
\end{theorem}
\noindent Condition $(\gamma_2)$ means that $(M^{1/2}_p)_{p \in \N}$ is strongly non-quasianalytic \cite{Komatsu}. 
The main goal of this article is to improve and complete Theorem \ref{prev-char} in the following three ways: Consider the Beurling case as well; replace $\operatorname{(slc)}$ and $\operatorname{(mg)}$ by the weaker conditions $\operatorname{(lc)}$ and $\operatorname{(dc)}$; characterize the existence of a continuous linear right inverse of $\mathcal{M}: \Sb_{\ast}(0,\infty) \rightarrow \Lambda_{\ast}$. More precisely, we show the following result; see Theorem \ref{main-theorem}.
\begin{theorem}\label{main-theorem-intro} Let $(M_p)_{p \in \N}$ be a weight sequence satisfying $(\operatorname{lc})$ and $(\operatorname{dc})$.
\begin{itemize}
\item[$(a)$] The following statements are equivalent:
\begin{itemize}
\item[$(i)$] $\mathcal{M}: \mathcal{S}_{(M_p)}(0,\infty) \rightarrow \Lambda_{(M_p)}$ is surjective.
\item[$(ii)$] $\mathcal{M}: \mathcal{S}_{(M_p)}(0,\infty) \rightarrow \Lambda_{(M_p)}$ has a continuous linear right inverse.
\item[$(iii)$] $(M_p)_{p \in \N}$ satisfies $(\gamma_2)$.
\end{itemize}
\item[$(b)$] $\mathcal{M}: \Sb_{\{M_p\}}(0,\infty) \rightarrow \Lambda_{\{M_p\}}$ is surjective if and only if $(M_p)_{p \in \N}$ satisfies $(\gamma_2)$.
\item[$(c)$]  $\mathcal{M}: \Sb_{\{M_p\}}(0,\infty) \rightarrow \Lambda_{\{M_p\}}$ has a continuous linear right inverse if and only if $(M_p)_{p \in \N}$ satisfies $(\gamma_2)$ and 
\begin{itemize}
\item[$(\beta_2)$] $\displaystyle \forall \varepsilon > 0 \, \exists n \in \Z_+ \, : \, \limsup_{p \to \infty}  \left(\frac{M_{np}}{M_p}\right)^{\frac{1}{p(n-1)}} \frac{M_{np-1}}{M_{np}} \leq \varepsilon$.
\end{itemize}
\end{itemize}
\end{theorem}
\noindent Condition $(\beta_2)$ is due to Petzsche \cite{Petzsche} and appears in his characterization of the existence of a continuous linear right inverse of the Borel mapping on spaces of  ultradifferentiable functions of Roumieu type. We also give an analogue of Theorem \ref{main-theorem-intro} for Gelfand-Shilov spaces of type $\mathcal{S}^\dagger_\ast(0,\infty)$ (cf.\ \cite[Thm.\ 3.5]{D-J-S}); see Theorem \ref{main-theorem-1}.

In \cite{D-J-S}, Theorem \ref{prev-char} is shown by reducing it to the Borel-Ritt problem \cite{Ramis, S-V,Thilliez, JG-S-S} in spaces of ultraholomorphic functions on the upper half-plane and then using  solutions to this problem from \cite{Thilliez, JG-S-S}; a technique that goes back to A. L.  Dur\'an and Estrada \cite{D-E}. Up until now,  this  seems to be the only known method to study the Stieltjes moment problem in  Gelfand-Shilov spaces. It also explains why we had to assume $\operatorname{(slc)}$ and $\operatorname{(mg)}$ in Theorem \ref{prev-char}: These conditions are needed to solve the Borel-Ritt problem in spaces of ultraholomorphic functions \cite[Thm.\ 3.2.1]{Thilliez}. We develop here a completely new approach. Namely, we show Theorem \ref{main-theorem-intro}  by reducing it to the Borel problem in spaces of ultradifferentiable functions of class $(N_p)_{p \in \N}$, where $(N_p)_{p \in \N}$ denotes the $2$-interpolating sequence associated to $(M_p)_{p \in \N}$ \cite{S-V}, and then using  Petzsche's  classical solution to this problem \cite{Petzsche}.

As a corollary, we obtain an analogue of Theorem \ref{main-theorem-intro} for the Borel-Ritt problem in spaces of ultraholomorphic functions on the upper half-plane, thereby improving some results of Schmets and Valdivia  \cite{S-V} and Thilliez \cite{Thilliez} in the particular case of the upper half-plane; see Theorem \ref{main-theorem-2} and Remark \ref{details}. Of course, due to the distinct geometry of the upper half-plane, this special case is much simpler to handle than the Borel-Ritt problem in spaces of ultraholomorphic functions on general sectors.

The plan of this article is as follows. In Section \ref{sect-prel}, we fix the notation, introduce weight sequences and recall Petzsche's solution to the Borel problem in spaces of ultradifferentiable functions. In Section \ref{sect-GS}, we define  Gelfand-Shilov spaces of type $\mathcal{S}_\ast$ and collect several properties of these spaces that will be used later on. Next, in the auxiliary Sections \ref{sect-FA} and \ref{sub-Eidelheit}, we present an abstract result about the existence of a continuous linear right inverse and prove a Borel type theorem. These results are used in the proof of Theorem \ref{main-theorem-intro}, which is  given in Section \ref{sect-main}. Finally, in Section \ref{sect-cor}, we consider the Stieltjes moment problem in Gelfand-Shilov spaces of type $\mathcal{S}^\dagger_\ast(0,\infty)$ and the Borel-Ritt problem in spaces of ultraholomorphic functions on the upper half-plane.

\section{Preliminaries}\label{sect-prel}
\subsection{Notation} 

We set $\N = \{0,1,2, \ldots\}$ and $\Z_+ = \{1,2, \ldots \}$. The Fr\'echet space of rapidly decreasing smooth functions on $\R$ is denoted by $\mathcal{S}(\R)$. We fix the constants in the Fourier transform as follows
$$
\mathcal{F}(\varphi)(\xi) = \widehat{\varphi}(\xi) :=  \int_{-\infty}^\infty \varphi(x) e^{ix\xi} \dx, \qquad \varphi \in \mathcal{S}(\R).
$$
The $p$-th moment, $p \in \N$, of an element $\varphi \in \mathcal{S}(\R)$ is given by
$$
\mu_p(\varphi) := \int_{-\infty}^\infty x^p \varphi(x)  \dx.
$$
Notice that $\widehat{\varphi}^{(p)}(0) = i^p \mu_p(\varphi)$ for all $p \in \N$.  

We define the \emph{Borel mapping}  as
$$
\mathcal{B}: C^\infty(\R) \rightarrow \C^\N: \varphi \rightarrow (\varphi^{(p)}(0))_{p \in \N}
$$
and the \emph{Stieltjes moment mapping} as
$$
\mathcal{M}: \mathcal{S}(\R) \rightarrow \C^\N: \varphi \rightarrow (\mu_p(\varphi))_{p \in \N}.
$$

A lcHs (= locally convex Hausdorff space) $E$ is said to be an \emph{$(LF)$-space} if there is a sequence $(E_n)_{n \in \N}$ of Fr\'echet spaces with $E_n \subseteq E_{n + 1}$ and continuous inclusion mappings for all $n \in \N$ such that $E = \bigcup_{n \in \N} E_n$  and  the topology of $E$ coincides with the finest locally convex topology such that all the inclusion mappings $E_n \rightarrow E$, $n \in \N$, are continuous. We write $E = \varinjlim_{n \in \N} E_n$. If the sequence $(E_n)_{n \in \N}$ consists of Banach spaces, $E$ is called an \emph{$(LB)$-space}. Finally, a lcHs is said to be a \emph{(PLB)}-space if it can be written as the projective limit of a countable spectrum of $(LB)$-spaces.

\subsection{Weight sequences}
A sequence $(M_p)_{p \in \N}$ of positive numbers is called  a \emph{weight sequence} if  $M_0 =1$ and $m_p := M_p/M_{p-1} \rightarrow \infty$ as $p \to \infty$. The \emph{associated function} of a weight sequence $(M_p)_{p \in \N}$ is defined as $M(0):=0$ and
$$
M(t):=\sup_{p\in\N}\log\frac{t^p}{M_p},\qquad t > 0.
$$
We will make use of the following conditions on weight sequences:
\begin{enumerate}
\item[$(\operatorname{lc})$] \emph{(log-convexity)} $M_p^2 \leq M_{p-1}M_{p+1}$, $p \in \Z_+$.
\item[$(\operatorname{dc})$] \emph{(derivation-closedness)} $M_{p+1} \leq C_0H^{p+1}M_p$, $p \in \N$, for some $C_0,H \geq 1$.
\item[$(\operatorname{mg})$] \emph{(moderate growth)} $M_{p+q} \leq C_0H^{p+q}M_pM_q$, $p,q \in \N$, for some $C_0,H \geq 1$.
\item[$(\gamma)$] \emph{(non-quasianalyticity)} $\displaystyle \sum_{p = 1}^\infty \frac{1}{m_p}< \infty$.
\item[$(\gamma_1)$] \emph{(strong non-quasianalyticity)} $\displaystyle \sup_{p \in \Z_+}\frac{m_p}{p}\sum_{q = p}^\infty \frac{1}{m_q}< \infty$.
\item[$(\gamma_2)$] $\displaystyle \sup_{p \in \Z_+}\frac{m_p^{1/2}}{p}\sum_{q = p}^\infty \frac{1}{m_q^{1/2}}< \infty$.
\item[$(\beta_2)$]  $\displaystyle \forall \varepsilon > 0 \, \exists n \in \Z_+ \, : \, \limsup_{p \to \infty}  \left(\frac{M_{np}}{M_p}\right)^{\frac{1}{p(n-1)}} \frac{1}{m_{np}} \leq \varepsilon$.
\end{enumerate}
Clearly, $(\operatorname{mg}) \Rightarrow (\operatorname{dc})$ and $(\gamma_1) \Rightarrow (\gamma)$. Moreover, if $(M_p)_{p \in \N}$ satisfies $(\operatorname{lc})$, then $(\gamma_2) \Rightarrow (\gamma_1)$. The conditions $(\operatorname{lc})$, $(\operatorname{dc})$, $(\operatorname{mg})$, $(\gamma)$ and $(\gamma_1)$  are standard in the theory of ultradifferentiable functions  and their meaning is well explained in the classical work of Komatsu \cite{Komatsu}. Conditions $(\gamma_1)$ and $(\gamma_2)$ are particular instances of 
\begin{enumerate}
\item[$(\gamma_r)$] $\displaystyle \sup_{p \in \Z_+}\frac{m^{1/r}_p}{p}\sum_{q = p}^\infty \frac{1}{m^{1/r}_q}< \infty, \qquad r > 0$. 
\end{enumerate}
Condition $(\gamma_r)$ means that $(M^{1/r}_p)_{p \in \N}$ satisfies $(\gamma_1)$. These conditions, which were introduced by Schmets and Valdivia \cite{S-V}  for $r \in \N$ and by Thilliez \cite{Thilliez} for arbitrary $r > 0$, play an important role in the study of the Borel-Ritt problem in spaces of ultraholomorphic functions \cite{Ramis, S-V, Thilliez, JG-S-S}. Condition $(\beta_2)$ is due to Petzsche \cite{Petzsche} and appears in his characterization of the existence of a continuous linear right inverse of the Borel mapping on spaces of  ultradifferentiable functions of Roumieu type;  see Theorem \ref{Borel-P} below.

\begin{remark} Consider
\begin{itemize}
\item[$(\beta^0_2)$] $\displaystyle \exists n \in \Z_+ \, : \, \lim_{p \to \infty} \frac{m_{np}}{m_p} = \infty$.
\item[$(\beta^1_2)$] $\displaystyle \lim_{p \to \infty} \frac{M_p^{1/p}}{m_p} = 0$.
\end{itemize}
Petzsche has shown that $(\beta^0_2) \Rightarrow (\beta_2) \Rightarrow (\beta^1_2)$ \cite[Prop.\ 1.5(b) and Prop.\ 1.6(a)]{Petzsche} and that the converse implications are false in general \cite[Example 1.8]{Petzsche}. However, $(\beta^0_2)$ and $(\beta_2)$ are equivalent within the class of weight sequences $(M_p)_{p \in \N}$ satisfying the following mild regularity condition: There is $n \in \Z_+$ such that the set of finite limit points of the set $\{ m_{n^l}/ m_{n^{l-1}} \, | \, l \in \Z_+\}$ is bounded  \cite[Prop.\ 1.6(b)]{Petzsche}. 
\end{remark}

\begin{example}
$(i)$ The Gevrey sequence $(p!^\alpha)_{p \in \N}$, $\alpha > 0$,  satisfies $(\operatorname{lc})$ and $(\operatorname{mg})$; it satisfies $(\gamma_r)$ if and only if  $\alpha > r$; it does not satisfy $(\beta^1_2)$ and, thus, also not $(\beta_2)$.

\noindent $(ii)$ The $q$-Gevrey sequence $(q^{p^2})_{p \in \N}$, $q >1$, satisfies $(\operatorname{lc})$ and $(\operatorname{dc})$ but not $(\operatorname{mg})$; it satisfies $(\gamma_r) > 0$ for all $r > 0$; it satisfies $(\beta^0_2)$ and, thus, also $(\beta_2)$.
\end{example}

Following \cite{S-V}, we define the \emph{$2$-interpolating sequence} $(N_p)_{p \in \N}$ associated to a weight sequence $(M_p)_{p \in \N}$ as 
$$
N_p := \left\{
	\begin{array}{ll}
		M_{q}, &  \mbox{$p = 2q$, $q \in \N$}, \\ \\
		(M_qM_{q+1})^{1/2},  &    \mbox{$p = 2q+1$, $q \in \N$}.
	\end{array}
\right. 
$$
\begin{lemma}\label{2-interpolating}  \cite[Lemma 2.3]{S-V}
Let $(M_p)_{p \in \N}$ be a weight sequence satisfying $(\operatorname{lc})$. Denote by $(N_p)_{p \in \N}$ its $2$-interpolating sequence. Then, $(N_p)_{p \in \N}$ is a weight sequence satisfying $(\operatorname{lc})$. Moreover, the following statements hold:
\begin{itemize}
\item[$(a)$] $(M_p)_{p \in \N}$ satisfies $(\operatorname{dc})$ if and only if $(N_p)_{p \in \N}$ does so.
\item[$(b)$] $(M_p)_{p \in \N}$ satisfies $(\gamma_2)$ if and only if $(N_p)_{p \in \N}$ satisfies $(\gamma_1)$.
\item[$(c)$] $(M_p)_{p \in \N}$ satisfies $(\beta_2)$ if and only if $(N_p)_{p \in \N}$ does so.
\end{itemize}
\end{lemma}
\subsection{The Borel problem in spaces of ultradifferentiable functions}
Let $(N_p)_{p \in \N}$ be a weight sequence. For $h > 0$ we define $\mathcal{D}^{N_p,h}_{[-1,1]}$ as the Banach space consisting of all $\varphi \in C^\infty(\R)$ with $\operatorname{supp} \varphi \subseteq [-1,1]$ such that
$$
\| \varphi \|_{\mathcal{D}^{N_p,h}_{[-1,1]}} := \sup_{p \in \N} \max_{x \in [-1,1]} \frac{h^p |\varphi^{(p)}(x)|}{M_p} < \infty.
$$ 
We set 
$$
\mathcal{D}^{(N_p)}_{[-1,1]} := \varprojlim_{h \to \infty} \mathcal{D}^{N_p,h}_{[-1,1]}, \qquad \mathcal{D}^{\{N_p\}}_{[-1,1]} := \varinjlim_{h \to 0^+} \mathcal{D}^{N_p,h}_{[-1,1]}.
$$
$\mathcal{D}^{(N_p)}_{[-1,1]}$ is a Fr\'echet space, while $\mathcal{D}^{\{N_p\}}_{[-1,1]}$ is an $(LB)$-space. If $(N_p)_{p \in \N}$ satisfies $(\operatorname{lc})$, the spaces $\mathcal{D}^{(N_p)}_{[-1,1]}$ and $\mathcal{D}^{\{N_p\}}_{[-1,1]}$ are non-trivial if and only if $(N_p)_{p \in \N}$ satisfies $(\gamma)$, as follows from the Denjoy-Carleman theorem.

For $h > 0$ we define $\Lambda_{N_p,h}$ as the Banach space consisting of all sequences $a = (a_p)_{p \in \N} \in \C^\N$ such that 
$$
\| a \|_{\Lambda_{N_p,h}} := \sup_{p \in \N} \frac{h^p|a_p|}{N_p} < \infty.
$$ 
We set
$$
\Lambda_{(N_p)} := \varprojlim_{h \to \infty}\Lambda_{N_p,h}, \qquad \Lambda_{\{N_p\}} := \varinjlim_{h \to 0^+}\Lambda_{N_p,h}.
$$
$\Lambda_{(N_p)}$ is a Fr\'echet space, while $\Lambda_{\{N_p\}}$ is an $(LB)$-space. The mappings 
$$
\mathcal{B}: \mathcal{D}^{(N_p)}_{[-1,1]}  \rightarrow \Lambda_{(N_p)}, \qquad \mathcal{B}: \mathcal{D}^{\{N_p\}}_{[-1,1]}   \rightarrow \Lambda_{\{N_p\}}
$$ 
are well-defined and continuous. Petzsche characterized the surjectivity and the existence of a continuous linear right inverse of these mappings in the following way.
\begin{theorem}  \label{Borel-P} Let $(N_p)_{p \in \N}$ be a weight sequence satisfying $(\operatorname{lc})$ and $(\gamma)$. 
\begin{itemize}
\item[$(a)$]  \emph{(\cite[Thm.\ 3.4]{Petzsche})} The following statements are equivalent:
\begin{itemize}
\item[$(i)$] $(N_p)_{p \in \N}$ satisfies $(\gamma_1)$.
\item[$(ii)$] $\mathcal{B}: \mathcal{D}^{(N_p)}_{[-1,1]}  \rightarrow \Lambda_{(N_p)}$
has a continuous linear right inverse.
\item[$(iii)$]  $\mathcal{B}: \mathcal{D}^{(N_p)}_{[-1,1]}  \rightarrow \Lambda_{(N_p)}$
is surjective.
\end{itemize}
\item[$(b)$]  \emph{(\cite[Thm.\ 3.5]{Petzsche})} $(N_p)_{p \in \N}$ satisfies $(\gamma_1)$ if and only if $\mathcal{B}: \mathcal{D}^{\{N_p\}}_{[-1,1]}  \rightarrow \Lambda_{\{N_p\}}$ is surjective.
\item[$(c)$]  \emph{(\cite[Thm.\ 3.1(a)]{Petzsche})}\footnote{As pointed out in \cite[p.\ 223]{S-V}, the statement of \cite[Thm.\ 3.1(a)]{Petzsche} contains a mistake, namely, one should read ``$(\gamma_1)$ and $(\beta_2)$" instead of ``$(\beta_2)$".}  $(N_p)_{p \in \N}$ satisfies $(\gamma_1)$ and $(\beta_2)$  if and only if $\mathcal{B}: \mathcal{D}^{\{N_p\}}_{[-1,1]}  \rightarrow \Lambda_{\{N_p\}}$ has a continuous linear right inverse.
\end{itemize}
\end{theorem}

\section{Gelfand-Shilov spaces of type $\mathcal{S}_\ast$}\label{sect-GS}
Let $(M_p)_{p \in \N}$ be a weight sequence. For $n \in \N$ and $h > 0$ we write $\mathcal{S}^n_{M_p,h}(\R)$ for the Banach space consisting of all $\varphi \in C^{n}(\R)$ such that
$$
\| \varphi\|_{\mathcal{S}^n_{M_p,h}} := \max_{m \leq n} \sup_{p \in \N} \sup_{x \in \R} \frac{h^p |x^p\varphi^{(m)}(x)|}{M_p} < \infty.
$$
Notice that
$$
\| \varphi\|_{\mathcal{S}^n_{M_p,h}} = \max_{m \leq n} \sup_{x \in \R} |\varphi^{(m)}(x)|e^{M(h|x|)}, \qquad \varphi \in \mathcal{S}^n_{M_p,h}(\R).
$$
We set
\begin{gather*}
\mathcal{S}_{(M_p)}(\R) := \varprojlim_{n \to \infty}\mathcal{S}^n_{M_p,n}(\R), \\
\Si_{\{M_p\}}(\R) := \varinjlim_{h \to 0^+}\varprojlim_{n \to \infty}\mathcal{S}^n_{M_p,h}(\R),\\
\Sp_{\{M_p\}}(\R) := \varprojlim_{n \to \infty}\varinjlim_{h \to 0^+}\mathcal{S}^n_{M_p,h}(\R).
\end{gather*}
$\mathcal{S}_{(M_p)}(\R)$ is a Fr\'echet space, $\Si_{\{M_p\}}(\R)$ is an $(LF)$-space, while $\Sp_{\{M_p\}}(\R)$ is a $(PLB)$-space. 
In the sequel, we shall sometimes use $\mathcal{S}_\ast(\R)$ as a common notation for $\mathcal{S}_{(M_p)}(\R)$, $\Si_{\{M_p\}}(\R)$ and $\Sp_{\{M_p\}}(\R)$; a similar convention will be used for other spaces. If $(M_p)_{p \in \N}$ satisfies $(\operatorname{dc})$, the mapping
$$
\mathcal{M} : \mathcal{S}_{\ast}(\R)  \rightarrow \Lambda_{\ast}
$$
is well-defined and continuous. The following result will be used later on.
\begin{proposition} \label{complete}
The $(LF)$-space $\Si_{\{M_p\}}(\R)$ is complete.
\end{proposition} 
\begin{proof}
In the notation of \cite{D-V}, we have that $\Si_{\{M_p\}}(\R) = \mathcal{B}_{\mathcal{V}}(\R)$, where $\mathcal{V} = (v_N)_{N \in \N}$ with $v_N = e^{M(|\, \cdot \,|/N)}$ for $N \in \N$. By \cite[Thm.\ 3.4]{D-V}, it suffices to show that $\mathcal{V}$ satisfies $(\Omega)$, that is,
\begin{gather*}
\forall N \, \exists L \geq N \, \forall K \geq L \, \exists \theta \in (0,1) \, \exists C > 0 \, \forall x \in \R \, : \\
 v_L(x) \leq C (v_N(x))^{1-\theta} (v_K(x))^\theta.
\end{gather*}
The latter follows from the fact that the function $t \rightarrow M(e^t)$ is increasing and convex on $[0,\infty)$.
\end{proof}
Next, we discuss the Fourier transform on $\mathcal{S}_\ast(\R)$. For $n \in \N$ and $h > 0$ we write $\mathcal{S}^{M_p,h}_n(\R)$ for the Banach space consisting of all $\varphi \in C^{\infty}(\R)$ such that
$$
\| \varphi\|_{\mathcal{S}_n^{M_p,h}} := \max_{m \leq n}\sup_{p \in \N} \sup_{x \in \R} \frac{h^p |x^m\varphi^{(p)}(x)|}{M_p} < \infty.
$$
We set
\begin{gather*}
\mathcal{S}^{(M_p)}(\R) := \varprojlim_{n \to \infty}\mathcal{S}_n^{M_p,n}(\R), \\
\Si^{\{M_p\}}(\R) := \varinjlim_{h \to 0^+}\varprojlim_{n \to \infty}\mathcal{S}_n^{M_p,h}(\R),\\
\Sp^{\{M_p\}}(\R) := \varprojlim_{n \to \infty}\varinjlim_{h \to 0^+}\mathcal{S}^{M_p,h}_n(\R).
\end{gather*}
$\mathcal{S}^{(M_p)}(\R)$ is a Fr\'echet space, $\Si^{\{M_p\}}(\R)$ is an $(LF)$-space, while $\Sp^{\{M_p\}}(\R)$ is a $(PLB)$-space.  
If $(M_p)_{p \in \N}$ satisfies $(\operatorname{lc})$ and $(\operatorname{dc})$, the Fourier transform is a topological isomorphism from $\mathcal{S}_{\ast}(\R)$ onto $\mathcal{S}^{\ast}(\R)$ (cf.\ \cite[Sect.\ IV.6]{G-S}). 

We now introduce Gelfand-Shilov spaces of type $\mathcal{S}_\ast(0,\infty)$. Let $n \in \N$ and $h > 0$. We define the following closed subspaces of $\mathcal{S}^n_{M_p,h}(\R)$
\begin{gather*}
\mathcal{S}^n_{M_p,h}(0,\infty) := \{ \varphi \in \mathcal{S}^n_{M_p,h}(\R) \, | \, \operatorname{supp} \varphi \subseteq  [0, \infty) \}, \\
\mathcal{S}^{n,0}_{M_p,h}(\R) :=  \{ \varphi \in \mathcal{S}^n_{M_p,h}(\R)  \, | \, \varphi^{(m)}(0) = 0 \mbox{ for all $m = 0, \ldots, n$} \}, 
\end{gather*}
and endow them with the norm $\| \, \cdot \, \|_{\mathcal{S}^n_{M_p,h}}$.  Hence, they become Banach spaces. We set
\begin{gather*}
\mathcal{S}_{(M_p)}(0,\infty) := \varprojlim_{n \to \infty} \mathcal{S}^n_{M_p,n}(0,\infty), \quad \mathcal{S}^0_{(M_p)}(\R) := \varprojlim_{n \to \infty} \mathcal{S}^{n,0}_{M_p,n}(\R), \\
\Si_{\{M_p\}}(0,\infty) := \varinjlim_{h \to 0^+}\varprojlim_{n \to \infty} \mathcal{S}^n_{M_p,h}(0,\infty), \quad \Si^0_{\{M_p\}}(\R) := \varinjlim_{h \to 0^+}\varprojlim_{n \to \infty} \mathcal{S}^{n,0}_{M_p,h}(\R), \\
\Sp_{\{M_p\}}(0,\infty) := \varprojlim_{n \to \infty}\varinjlim_{h \to 0^+} \mathcal{S}^n_{M_p,h}(0,\infty), \quad \Sp^0_{\{M_p\}}(\R) := \varprojlim_{n \to \infty}\varinjlim_{h \to 0^+}\mathcal{S}^{n,0}_{M_p,h}(\R).
\end{gather*}
Notice that $\Si_{\{M_p\}}(0,\infty)$ was denoted by $\Sb_{\{M_p\}}(0,\infty)$ in the introduction. $\mathcal{S}_{(M_p)}(0,\infty)$ and $\mathcal{S}^0_{(M_p)}(\R)$ are Fr\'echet spaces, $\Si_{\{M_p\}}(0,\infty)$ and $\Si^0_{\{M_p\}}(\R)$ are $(LF)$-spaces, while $\Sp_{\{M_p\}}(0,\infty)$ and  $\Sp^0_{\{M_p\}}(\R)$ are $(PLB)$-spaces. We have that
\begin{gather}
\mathcal{S}_{\ast}(0,\infty) = \{ \varphi \in \mathcal{S}_{\ast}(\R)  \, | \, \operatorname{supp} \varphi \subseteq  [0, \infty) \}, \label{eq-1}  \\
\mathcal{S}_{\ast}^0(\R) =  \{ \varphi \in  \mathcal{S}_{\ast}(\R)  \, | \, \varphi^{(n)}(0) = 0 \mbox{ for all $n \in \N$} \} \label{eq-2},
\end{gather}
as sets. 
\begin{lemma} \label{ind-proj-description}  
Let $\mathcal{S}_{\ast}(\R) = \mathcal{S}_{(M_p)}(\R)$ or $\mathcal{S}_{\ast}(\R) = \Si_{\{M_p\}}(\R)$. Then, the equalities \eqref{eq-1} and \eqref{eq-2} hold topologically if the spaces at the right-hand side are endowed with the relative topology induced by $\mathcal{S}_{\ast}(\R)$. 
\end{lemma}
We need some preparation for the proof of Lemma \ref{ind-proj-description}. Let $E = \varinjlim_{n \in \N} E_n$ be an $(LF)$-space. A subspace $L$ of $E$ is called a \emph{limit subspace of E} if $L = \varinjlim_{n \in \N} L \cap E_n$ topologically, where $L$ is endowed with the relative topology induced by $E$ and $L \cap E_n$, $n \in \N$, is endowed with the relative topology induced by $E_n$. The following result is a consequence of \cite[Prop.\ 1.2]{Vogt} and the fact that every Fr\'echet space is an acyclic $(LF)$-space; we refer to \cite{Vogt} for the definition of an acyclic $(LF)$-space.
\begin{lemma}\label{limit-1} \emph{(cf.\ \cite[Prop.\ 1.2]{Vogt})}
Let $E$ be an $(LF)$-space, let $F$ be a Fr\'echet space and let $T: E \rightarrow F$ be a surjective continuous linear mapping. Then, $\operatorname{ker} T$ is a limit subspace of $E$. 
\end{lemma}

\begin{lemma}\label{UB}
Let $E$ be an $(LF)$-space. Every complemented subspace of $E$ is a limit subspace of $E$.
\end{lemma}
\begin{proof}
Since $(LF)$-spaces are webbed and ultrabornological \cite[Remark 24.36]{M-V} and the class of ultrabornological lcHs is closed under taking complemented subspaces, this follows from De Wilde's open mapping theorem \cite[Thm.\ 24.30]{M-V}.
\end{proof}
\begin{proof}[Proof of Lemma \ref{ind-proj-description}]
$\mathcal{S}_{\ast}(\R) = \mathcal{S}_{(M_p)}(\R)$: Obvious.

$\mathcal{S}_{\ast}(\R) = \Si_{\{M_p\}}(\R)$: It suffices to show that $\Si_{\{M_p\}}(0,\infty)$ and $\Si^0_{\{M_p\}}(\R)$ are limit subspaces of $\Si_{\{M_p\}}(\R)$.  We first consider $\Si^0_{\{M_p\}}(\R)$. By Borel's theorem, the continuous linear mapping $\mathcal{B}: \Si_{\{M_p\}}(\R) \rightarrow \C^\N$ is surjective. Clearly,  $\Si^0_{\{M_p\}}(\R) = \ker \mathcal{B}$. Hence, the result follows from Lemma \ref{limit-1}. Next, we deal with $\Si_{\{M_p\}}(0,\infty)$. Since $\Si_{\{M_p\}}(0,\infty)$ is a complemented subspace of $\Si^0_{\{M_p\}}(\R)$,  Lemma \ref{UB} yields that 
$\Si_{\{M_p\}}(0,\infty)$ is a limit subspace of $\Si^0_{\{M_p\}}(\R)$. Consequently, as we already have shown that $\Si^0_{\{M_p\}}(\R)$ is a limit subspace of $\Si_{\{M_p\}}(\R)$, $\Si_{\{M_p\}}(0,\infty)$ is a limit subspace of $\Si_{\{M_p\}}(\R)$.
\end{proof}
Finally, we present two technical lemmas that will play an important role later on.
\begin{lemma}\label{reduction-2-0}
Let $(N_p)_{p \in \N}$ be a weight sequence satisfying $(\operatorname{lc})$ and $(\operatorname{dc})$.
\begin{itemize}
\item[$(a)$] $T: \mathcal{S}_{(N_p)}(0,\infty) \rightarrow \mathcal{S}_{(N_p)}(0,\infty)$ and  $T: \Si_{\{N_p\}}(0,\infty) \rightarrow \Si_{\{N_p\}}(0,\infty)$ are well-defined continuous mappings, where 
$$
T(\varphi)(x)=
\left\{
	\begin{array}{ll}
		\frac{\varphi(x)}{x},  &  \mbox{$x > 0$}, \\ \\
		0,  &   \mbox{$x \leq 0$}.
	\end{array}
\right. 
$$
\item[$(b)$] $T: \mathcal{S}_{(N_p)}(0,\infty) \rightarrow \mathcal{S}_{(N_p)}(0,\infty)$ and  $T: \Sp_{\{N_p\}}(0,\infty) \rightarrow \Sp_{\{N_p\}}(0,\infty)$ are well-defined continuous  mappings, where 
$$
T(\varphi)(x)=
\left\{
	\begin{array}{ll}
		x\varphi(x),  &  \mbox{$x > 0$}, \\ \\
		0,  &   \mbox{$x \leq 0$}.
	\end{array}
\right. 
$$
\end{itemize}
\end{lemma}
\begin{proof} We start by recalling the following consequence of Taylor's theorem: Let $\varphi \in C^n([0,1])$, $n \in \N$, be such that $\varphi^{(m)}(0) = 0$ for all $m = 0, \ldots, n$. Then,
\begin{equation}
|\varphi^{(j)}(x)| \leq\frac{\|\varphi^{(j+k)}\|_{L^\infty([0,1])}}{k!} x^k, \qquad x \in [0,1],
\label{taylor}
\end{equation}
for all $j,k \in \N \mbox{ with } j + k \leq n$. 

$(a)$ It suffices to show that $T : \mathcal{S}^{n+2}_{N_p,h}(0,\infty) \rightarrow \mathcal{S}^{n}_{N_p,h}(0,\infty)$ is well-defined and continuous for all $n \in \N$ and $h > 0$. Let $\varphi \in  \mathcal{S}^{n+2}_{N_p,h}(0,\infty)$ be arbitrary. It holds that
$$
T(\varphi)^{(m)}(x) = \sum_{j=0}^m \binom{m}{j} (-1)^j j!\frac{\varphi^{(m-j)}(x)}{x^{j+1}}, \qquad x > 0,
$$
for all $m \leq n$. Hence, \eqref{taylor} yields that $T(\varphi) \in C^n(\R)$ with $\operatorname{supp} \varphi \subseteq [0,\infty)$ and
\begin{align*}
\| T(\varphi) \|_{ \mathcal{S}^{n}_{N_p,h}} &\leq \max_{m \leq n} \sup_{p \in \N} \sup_{x > 0}  \frac{h^px^p}{N_p}  \sum_{j=0}^m \binom{m}{j} j!\frac{|\varphi^{(m-j)}(x)|}{x^{j+1}} \\
& \leq \max_{m \leq n} \sup_{p \in \N}\sup_{0 <x < 1}  \frac{h^p}{N_p}  \sum_{j=0}^m \binom{m}{j}  j!\frac{|\varphi^{(m-j)}(x)|}{x^{j+1}}  \\
&+  \max_{m \leq n} \sup_{p \in \N}\sup_{x \geq 1} \sum_{j=0}^m \binom{m}{j} j!\frac{h^px^p|\varphi^{(m-j)}(x)|}{N_p} \\
& \leq 2^ne^{N(h)}\max_{m \leq n} \| \varphi^{(m+1)}\|_{L^\infty([0,1])}+ 2^n n! \|\varphi \|_{ \mathcal{S}^{n}_{N_p,h}} \\
&\leq C \|\varphi \|_{\mathcal{S}^{n+2}_{N_p,h}}.
\end{align*}
$(b)$ It suffices to show that $T : \mathcal{S}^{n}_{N_p,hH}(0,\infty) \rightarrow \mathcal{S}^{n}_{N_p,h}(0,\infty)$ is well-defined and continuous for all $n \in \N$ and $h > 0$. Let $\varphi \in \mathcal{S}^{n}_{N_p,hH}(0,\infty)$ be arbitrary. It holds that $T(\varphi) \in C^n(\R)$ with $\operatorname{supp} \varphi \subseteq [0,\infty)$ and
\begin{align*}
\| T(\varphi) \|_{ \mathcal{S}^{n}_{N_p,h}} &\leq \max_{m \leq n} \sup_{p \in \N}\sup_{x > 0} \frac{h^px^{p+1}|\varphi^{(m)}(x)|}{N_p} +  \max_{m \leq n} \sup_{p \in \N}\sup_{x > 0} m  \frac{h^px^{p}|\varphi^{(m-1)}(x)|}{N_p} \\
&\leq  C \| \varphi \|_{ \mathcal{S}^{n}_{N_p,Hh}}. 
 \end{align*}
\end{proof}

\begin{lemma}\label{reduction-2-1}
Let $(M_p)_{p \in \N}$ be a weight sequence satisfying $(\operatorname{lc})$ and $(\operatorname{dc})$. Denote by $(N_p)_{p \in \N}$ its $2$-interpolating sequence.
\begin{itemize}
\item[$(a)$] $T: \mathcal{S}_{(N_p)}(0,\infty) \rightarrow \mathcal{S}_{(M_p)}(0,\infty)$ and $T: \Si_{\{N_p\}}(0,\infty) \rightarrow \Si_{\{M_p\}}(0,\infty)$  are well-defined continuous  mappings, where 
$$
T(\varphi)(x)=
\left\{
	\begin{array}{ll}
		\varphi(x^{1/2}),  &  \mbox{$x > 0$}, \\ \\
		0,  &   \mbox{$x \leq 0$}.
	\end{array}
\right. 
$$
\item[$(b)$] $T: \mathcal{S}_{(M_p)}(0,\infty) \rightarrow \mathcal{S}_{(N_p)}(0,\infty)$ and $T: \Sp_{\{M_p\}}(0,\infty) \rightarrow \Sp_{\{N_p\}}(0,\infty)$  are well-defined continuous mappings, where
$$
T(\varphi)(x)=
\left\{
	\begin{array}{ll}
		\varphi(x^2)  &  \mbox{$x > 0$}, \\ \\
		0,  &   \mbox{$x \leq 0$}.
	\end{array}
\right. 
$$
\end{itemize}
\end{lemma}
\begin{proof} 
$(a)$  It suffices to show that $T : \mathcal{S}^{2n+1}_{N_p,h^{1/2}}(0,\infty) \rightarrow \mathcal{S}^{n}_{M_p,h}(0,\infty)$ is well-defined and continuous for all $n \in \N$ and $h > 0$. Let $\varphi \in  \mathcal{S}^{2n+1}_{N_p,h^{1/2}}(0,\infty)$ be arbitrary. We set $I_0 = \{0\}$ and
$$
I_m = \{ \alpha = (\alpha_1, \ldots, \alpha _m) \in \N^m \, | \, \sum_{j=1}^m j \alpha_j = m \}, \qquad m \in \Z_+.
$$
Fa\`a di Bruno's formula implies that
$$
T(\varphi)^{(m)}(x) = \sum_{\alpha \in I_m} a_\alpha \frac{\varphi^{(|\alpha|)}(x^{1/2})}{x^{m - |\alpha|/2 }}, \qquad x > 0,
$$
for all $m \leq n$, where $a_{\alpha}$ are real constants. Hence, \eqref{taylor} yields that $T(\varphi) \in C^n(\R)$ with $\operatorname{supp} \varphi \subseteq [0,\infty)$. Since $M_p = N_{2p}$ for all $p \in \N$, it holds that
\begin{align*}
\| T(\varphi) \|_{ \mathcal{S}^{n}_{M_p,h}} &\leq \max_{m \leq n} \sup_{p \in \N}\sup_{x > 0} \frac{h^px^p}{M_p}\sum_{\alpha \in I_m} |a_{\alpha}| \frac{|\varphi^{(|\alpha|)}(x^{1/2})|}{x^{m - |\alpha|/2 }} \\
& \leq \max_{m \leq n}\sup_{p \in \N} \sup_{0 <x < 1}  \frac{h^p}{M_p}\sum_{\alpha \in I_m} |a_{\alpha}| \frac{|\varphi^{(|\alpha|)}(x^{1/2})|}{x^{m - |\alpha|/2 }}  \\
&+  \max_{m \leq n}\sup_{p \in \N} \sup_{x \geq 1}   \sum_{\alpha \in I_m} |a_{\alpha}| \frac{h^px^p|\varphi^{(|\alpha|)}(x^{1/2})|}{M_p} \\
&\leq e^{M(h)}\max_{m \leq n} \left(\sum_{\alpha \in I_m} |a_{\alpha}| \right) \| \varphi^{(2m)}\|_{L^\infty([0,1])} +   \left(\max_{m \leq n} \sum_{\alpha \in I_m} |a_{\alpha}| \right) \|\varphi \|_{ \mathcal{S}^{n}_{N_p,h^{1/2}}}\\
&\leq C \|\varphi \|_{ \mathcal{S}^{2n+1}_{N_p,h^{1/2}}}.
\end{align*}

$(b)$ Lemma \ref{2-interpolating}$(a)$ yields that $(N_p)_{p \in \N}$ satisfies $(\operatorname{dc})$. We may assume without loss of generality that the constants $C_0$ and $H$ occuring in $(\operatorname{dc})$ are the same for $(M_p)_{p \in \N}$ and $(N_p)_{p \in \N}$. It suffices to show that $T : \mathcal{S}^{n}_{M_p,H^{2n+1}h^2}(0,\infty) \rightarrow \mathcal{S}^{n}_{N_p,h}(0,\infty)$ is well-defined and continuous  for all $n \in \N$ and $h > 0$. Set $l = H^{2n+1}h^2$. Let $\varphi \in  \mathcal{S}^{n}_{M_p,l}(0,\infty)$ be arbitrary. Clearly, $T(\varphi) \in C^n(\R)$ with $\operatorname{supp} \varphi \subseteq [0,\infty)$. Fa\`a di Bruno's formula implies that
$$
T(\varphi)^{(m)}(x) = \sum_{j = 0}^{\lfloor m/2 \rfloor} a_j\varphi^{(m-j)}(x^2)x^{m-2j}, \qquad  x > 0,
$$
for all $m \leq n$, where $a_j$ are positive constants. It holds that $N_{p+q} \leq C_0^q H^{q(q-1)/2} H^{pq}N_p$ for all  $p,q \in \N$.  Moreover, there is $C_1 > 0$ such that $M_{\lceil p/2 \rceil} \leq C_1 H^{\lceil p/2 \rceil} N_p$ for all $p \in \N$. Therefore,
\begin{align*}
\| T(\varphi) \|_{ \mathcal{S}^{n}_{N_p,h}} &\leq  \max_{m \leq n} \sup_{p \in \N}\sup_{x > 0} \frac{h^px^p}{N_p} \sum_{j = 0}^{\lfloor m/2 \rfloor} a_j|\varphi^{(m-j)}(x^2)|x^{m-2j} \\
& \leq\max_{m \leq n} \sup_{p \in \N}\sup_{0 <x < 1} \frac{h^p}{N_p} \sum_{j = 0}^{\lfloor m/2 \rfloor} a_j |\varphi^{(m-j)}(x^2)| \\
&+ \max_{m \leq n}\sup_{p \in \N} \sup_{ x \geq 1} \sum_{j = 0}^{\lfloor m/2 \rfloor} a_j\frac{h^px^{p+m}|\varphi^{(m-j)}(x^2)|}{N_p}\\
&\leq e^{N(h)}\left(\max_{m \leq n}\sum_{j = 0}^{\lfloor m/2 \rfloor} a_j\right)\| \varphi \|_{ \mathcal{S}^{n}_{M_p,l}} \\
&+\max_{m \leq n} \left(\sum_{j = 0}^{\lfloor m/2 \rfloor} a_j\right)  \| \varphi \|_{ \mathcal{S}^{n}_{M_p,l}} \sup_{p \in \N} \frac{h^pM_{\lceil(p+m)/2\rceil}}{l^{\lceil(p+m)/2\rceil}N_p}\\
&\leq C\| \varphi \|_{ \mathcal{S}^{n}_{M_p,l}}. 
\end{align*}
\end{proof}

\section{A functional analytic tool}\label{sect-FA}
In this section, we show an abstract result about the existence of a continuous linear right inverse that it is tailor-made to prove Proposition \ref{moment-flat} below. We start with the following simple observation.
\begin{lemma}\label{surjectivity-kernel}
Let $E$, $F$ and $G$ be vector spaces  and let $T: E \rightarrow F$ and $S: E \rightarrow G$ be linear mappings. If  both $T: E \rightarrow F$  and  $S_{|\ker T} : \ker T \rightarrow G$ are surjective, then $T_{|\ker S} : \ker S \rightarrow F$ is also surjective.
\end{lemma}
\begin{proof}
 Let $x \in F$ be arbitrary. Choose $y \in E$ such that $T(y) = x$ and $z \in \ker T$  such that $S(z) = S(y)$. Then, $y-z \in \ker S$ and $T(y-z) = T(y) = x$.
\end{proof}
Now suppose that $E$, $F$ and $G$ are lcHs and that $T$ and $S$ are continuous linear mappings. If both $T$  and  $S_{|\ker T}$ have a continuous linear right inverse, it is clear from the proof of Lemma \ref{surjectivity-kernel} that also $T_{|\ker S}$ has a continuous linear right inverse. Our goal is to show that, under suitable extra conditions on $F$ and $T$,  $T_{|\ker S}$ has a continuous linear right inverse if one merely assumes that  $S_{|\ker T}$ lifts bounded sets. We need some preparation to formulate and prove this result. 

Let $E$ and $F$ be lcHs. We denote by $\csn(E)$ the set consisting of all continuous seminorms on $E$ and by $L(E,F)$ the space consisting of all continuous linear mappings from $E$ to $F$. 
Let $(x_n)_{n \in \N} \subset F$ and  $(x'_n)_{n \in \N} \subset F'$. The pair $((x_n)_{n \in \N}, (x'_n)_{n \in \N})$ is said to be a \emph{Schauder frame (in $F$)} \cite[Def.\ 1.1]{B-F-G-R} if 
$$
x = \sum_{n = 0}^\infty \langle x'_n,x \rangle x_n
$$  
for all $x \in F$. A Schauder frame $((x_n)_{n \in \N}, (x'_n)_{n \in \N})$ is called \emph{absolute} if for all $p \in \csn(F)$ there is $q \in \csn(F)$ such that
$$
\sum_{n = 0}^\infty |\langle x'_n,x \rangle|  p(x_n) \leq q(x)
$$
for all $x \in F$. Every absolute Schauder basis \cite[p.\ 340]{M-V} canonically determines a Schauder frame. We need the following lemma.
\begin{lemma}\label{right-inverse-absolute-Schauder-frame}
Let $E$ and $F$ be lcHs and let $T \in L(E,F)$. Suppose that $E$ is sequentially complete and that $F$ possesses an absolute Schauder frame  $((x_n)_{n \in \N}, (x'_n)_{n \in \N})$. If there is a sequence  $(y_n)_{n \in \N} \subset E$ such that $T(y_n) = x_n$ for all $n \in \N$ and for all $p \in \csn(E)$ there is $q \in \csn(F)$ such that $p(y_n) \leq q(x_n)$ for all $n \in \N$, then $T$ has a continuous linear right inverse.

\end{lemma}
\begin{proof}
For each $x \in F$ the sequence $\left( \sum_{n = 0}^N  \langle x'_n,x \rangle y_n  \right)_{N \in \N}$ is Cauchy in $E$. Since $E$ is sequentially complete, we have that 
$$
R(x) = \sum_{n = 0}^\infty  \langle x'_n,x \rangle y_n \in E.
$$
Then, $R: F \rightarrow E$ is a continuous linear right inverse of $T$.
\end{proof}
\begin{proposition}\label{functional-analysis}
Let $E$, $F$ and $G$ be lcHs  and let $T \in L(E,F)$ and $S \in L(E,G)$.  Endow $\ker T$ and $\ker S$ with the relative topology induced by $E$.  Suppose that the following conditions are satisfied:
\begin{itemize}
\item[$(1)$] $E$ is sequentially complete.
\item[$(2)$] $F$ possesses an absolute Schauder frame.
\item[$(3)$] $S_{|\ker T} : \ker T \rightarrow G$ lifts bounded sets, that is, for every $B \subset G$ bounded there is $A \subset \ker T$ bounded such that $S(A) = B$.
\item[$(4)$] There is a lcHs $E_0$ with the following properties:
\begin{itemize}
\item[$(4.1)$] $E_0 \subset E$ with continuous inclusion mapping.
\item[$(4.2)$] $T_{|E_0} : E_0 \rightarrow F$ has a continuous linear right inverse.
\item[$(4.3)$] $S_{|E_0} : E_0 \rightarrow G$ is locally bounded, that is, there is a neighbourhood $U$ of $0$ in $E_0$ such that $S(U)$ is bounded in $G$.
\end{itemize}
\end{itemize}
Then, $T_{| \ker S} : \ker S \rightarrow F$ has a continuous linear right inverse.
\end{proposition}
\begin{proof}
Suppose that $((x_n)_{n \in \N}, (x'_n)_{n \in \N})$ is an absolute Schauder frame in $F$. Since $\ker S$ is sequentially complete (as a closed subspace of $E$), it suffices to construct a sequence $(y_n)_{n \in \N} \subset \ker S$ satisfying the assumptions of Lemma \ref{right-inverse-absolute-Schauder-frame}. By $(4.2)$, there is a sequence $(y_{0,n})_{n \in \N} \subset E_0$  such that $T(y_{0,n}) = x_n$ for all $n \in \N$ and for all $p \in \csn(E_0)$ there is $q \in \csn(F)$ such that $p(y_{0,n}) \leq q(x_n)$ for all $n \in \N$. $(4.3)$ means that there is $p_0 \in \csn(E_0)$ such that $S_{|E_0} : (E_0,p_0) \rightarrow G$ is continuous, where $(E_0,p_0)$ stands for the vector space $E_0$ endowed with the topology generated by the single seminorm $p_0$. In particular, $S(y) = 0$ for all $y \in E_0$ with $p_0(y) = 0$. We set
$$
s_n =
\left\{
	\begin{array}{ll}
		\frac{S(y_{0,n})}{p_0(y_{0,n})},  &  \mbox{$p_0(y_{0,n}) \neq 0$}, \\ \\
		0,  &   \mbox{$p_0(y_{0,n}) = 0$},
	\end{array}
\right.
$$
for $n \in \N$. Since the sequence $(s_n)_{n \in \N}$ is bounded in $G$, $(3)$ yields that there is a bounded sequence $(z_n)_{n \in \N} \subset  \ker T$ such that $S(z_n) = s_n$ for all $n \in \N$. Set $y_n = y_{0,n} - p_0(y_{0,n})z_n$ for all $n \in \N$. Then, $y_n \in \ker S$ and $T(y_n) = T(y_{0,n}) = x_n$ for all $n \in \N$. Finally, let $p \in \csn(E)$ be arbitrary. Since the sequence $(z_n)_{n \in \N}$ is bounded in $E$, there is $C > 0$ such that
$$
p(y_n) \leq  p(y_{0,n}) + p_0(y_{0,n})p(z_n) \leq p_{|E_0}(y_{0,n}) + Cp_0(y_{0,n}) \leq  p'(y_{0,n}) 
$$
for all $n \in \N$, where $p' = (1+C) \max\{p_{|E_0}, p_0\}$. $(4.1)$ yields that $p' \in \csn(E_0)$. Hence, there is $q \in \csn(F)$ such that $p(y_n) \leq  p'(y_{0,n}) \leq q(x_n)$ for all $n \in \N$.
\end{proof}
\begin{remark}\label{Schwartz}
If $E$ is an $(FS)$-space (= Fr\'echet-Schwartz space) and $G$ is a Fr\'echet space, then condition $(3)$ in Proposition \ref{functional-analysis} may be relaxed to ``$S_{|\ker T} : \ker T \rightarrow G$ is surjective". Indeed, as a closed subspace of an $(FS)$-space is again an $(FS)$-space, $\ker (S_{|\ker T})$ is an $(FS)$-space. Since every  $(FS)$-space is quasinormable, the result follows from the fact that a surjective continuous linear mapping $Q: X \rightarrow Y$ between two Fr\'echet spaces $X$ and $Y$ lifts bounded sets if $\ker Q$ is quasinormable \cite[Lemma 26.13]{M-V}.
\end{remark}
\section{The Borel problem in  $\Sb_{(N_p),0}(\R)$}\label{sub-Eidelheit} 
Given a weight sequence $(N_p)_{p \in \N}$, we define the following closed subspace of $\mathcal{S}_{(N_p)}(\R)$
$$
\mathcal{S}_{(N_p),0}(\R) :=  \{ \varphi \in \mathcal{S}_{(N_p)}(\R)  \, | \, \mu_n(\varphi) = 0 \mbox{ for all $n \in \N$} \}
$$
and endow it with the relative topology induced by $\mathcal{S}_{(N_p)}(\R)$. Hence, it becomes a Fr\'echet space. The goal of this section is to show the following result; it will be used in the proof of Proposition \ref{moment-flat} below.
\begin{proposition}\label{Eidelheit-Borel}
Let $(N_p)_{p \in \N}$ be a weight sequence satisfying $(\operatorname{lc})$, $(\operatorname{dc})$ and $(\gamma_1)$. Then, 
$
\mathcal{B}: \mathcal{S}_{(N_p),0}(\R) \rightarrow \C^\N
$
 is surjective.
\end{proposition}
If $(N_p)_{p \in \N}$ satisfies $\operatorname{(lc)}$ and $\operatorname{(dc)}$, \cite[Prop.\ 3.4]{Komatsu} implies that $\mathcal{S}_{(N_p)}(\R)$ is an $(FS)$-space. Hence, Proposition \ref{Eidelheit-Borel}
can be strengthened as follows (cf.\ Remark \ref{Schwartz}).
\begin{proposition}\label{Eidelheit-Borel-1}
Let $(N_p)_{p \in \N}$ be a weight sequence satisfying $(\operatorname{lc})$, $(\operatorname{dc})$ and $(\gamma_1)$. Then, 
$
\mathcal{B}: \mathcal{S}_{(N_p),0}(\R) \rightarrow \C^\N
$
lifts bounded sets.
\end{proposition}

The proof of Proposition \ref{Eidelheit-Borel} is based on the following variant of Eidelheit's theorem. 
\begin{proposition}\label{Eidelheit-consequence} Let $E$ be a Fr\'echet space and let $(x'_n)_{n \in \N} \subset E'$. Let $F$ be a closed subspace of $E$ and set
$$
F^\perp = \{ x' \in E' \, | \, \langle x', x \rangle = 0 \mbox{ for all $x \in F$} \}.
$$
 The mapping
$
F \rightarrow \C^\N: x \rightarrow (\langle x'_n, x \rangle)_{n \in \N}
$
is surjective if and only if
\begin{itemize} 
\item[$(1)$] For all $N \in \N$ and $c_0, \ldots, c_N \in \C$ it holds that
$$
\sum_{n = 0}^N c_n x'_n \in  F^\perp
$$
implies that $c_0 = \cdots = c_N = 0$. 
\item[$(2)$] For every $B \subset E'$ equicontinuous there is $\nu \in \N$ such that for all $N \geq \nu$ and  $c_0, \ldots, c_N \in \C$ it holds that
$$
\sum_{n = 0}^N c_n x'_n \in  B + F^\perp
$$
implies that $c_{\nu} = \cdots = c_N = 0$. 
\end{itemize}
\end{proposition}
\begin{proof}
This is a consequence of the classical theorem of Eidelheit \cite[Thm.\ 26.27]{M-V} and the Hahn-Banach theorem.
\end{proof}
We shall prove Proposition \ref{Eidelheit-Borel} via Proposition \ref{Eidelheit-consequence} with $E = \mathcal{S}_{(N_p)}(\R)$, $F = \mathcal{S}_{(N_p),0}(\R)$  and  $(x'_n)_{n \in \N} = ((-1)^n\delta^{(n)})_{n \in \N}$). To this end, we first give an explicit description of the space $(\mathcal{S}_{(N_p),0}(\R))^\perp$ and the equicontinuous subsets of $\mathcal{S}'_{(N_p)}(\R)$. We need some preparation.

Let $(N_p)_{p \in \N}$ be a weight sequence.  An entire function $P(z) = \sum_{p = 0}^\infty b_pz^p$, $b_p \in \C$, is said to be an \emph{ultrapolynomial of class $(N_p)$} if 
\begin{equation}
\sup_{p \in \N} \frac{|b_p|N_p}{h^p} < \infty
\label{dual-sequence}
\end{equation}
for some $h > 0$. If $(N_p)_{p \in \N}$ satisfies $(\operatorname{lc})$, an entire function $P$ is an ultrapolynomial of class $(N_p)$ if and only if
$$
\sup_{z \in \C }|P(z)| e^{-N(h|z|)} < \infty
$$ 
for some $h > 0$.
\begin{lemma}\label{description-kernel}
Let $(N_p)_{p \in \N}$ be a weight sequence satisfying $(\operatorname{lc})$, $(\operatorname{dc})$ and $(\gamma_1)$. Then, $f \in \mathcal{S}'_{(N_p)}(\R)$ belongs to $(\mathcal{S}_{(N_p),0}(\R))^\perp$ if and only if there is an ultrapolynomial $P$ of class $(N_p)$ such that $f = P$ in $ \mathcal{S}'_{(N_p)}(\R)$.
\end{lemma}
\begin{proof}
Since $\mathcal{D}^{(N_p)}_{[-1,1]} \subseteq \mathcal{S}^{(N_p)}(\R)$, Theorem \ref{Borel-P}$(a)$ implies that 
$\mathcal{B}:  \mathcal{S}^{(N_p)}(\R) \rightarrow \Lambda_{(N_p)}$ is surjective. By taking the Fourier transform, we obtain that  $\mathcal{M}:  \mathcal{S}_{(N_p)}(\R) \rightarrow \Lambda_{(N_p)}$ is surjective. Consequently,  the sequence 
$$
 0 \longrightarrow \mathcal{S}_{(N_p),0}(\R) \xrightarrow{\phantom{,,}\iota\phantom{,,}} \mathcal{S}_{(N_p)}(\R) \xrightarrow{\phantom, \mathcal{M} \phantom,}    \Lambda_{(N_p)} \longrightarrow 0
$$
is exact, where $\iota: \mathcal{S}_{(N_p),0}(\R) \rightarrow \mathcal{S}_{(N_p)}(\R)$ denotes the inclusion mapping. Therefore, its dual sequence
$$
 0 \longrightarrow  \Lambda'_{(N_p)} \xrightarrow{\phantom{,}\mathcal{M}^t\phantom{,}} \mathcal{S}'_{(N_p)}(\R) \xrightarrow{\phantom{,,} \iota^t \phantom{,,}}  \mathcal{S}'_{(N_p),0}(\R)  \longrightarrow 0 
$$
is also exact \cite[Prop.\ 26.4]{M-V}. In particular, $\operatorname{im} \mathcal{M}^t = \ker \iota^t$. It is clear that $\ker \iota^t = (\mathcal{S}_{(N_p),0}(\R))^\perp$. On the other hand, $\Lambda'_{(N_p)}$ may be identified with the space consisting of all sequences $b= (b_p)_{p \in \N} \in \C^\N$ satisfying \eqref{dual-sequence} for some $h> 0$ and, under this identification, the duality is given by 
$$
\langle b, a \rangle = \sum_{p = 0}^\infty b_p a_p, \qquad a= (a_p)_{p \in \N} \in \Lambda_{(N_p)}, b = (b_p)_{p \in \N} \in \Lambda'_{(N_p)}.
$$
Hence, $\operatorname{im} \mathcal{M}^t$ coincides with the subspace of $\mathcal{S}'_{(N_p)}(\R)$ consisting of all ultrapolynomials of class $(N_p)$.
\end{proof}
The next result follows from the  structural theorem for general Gelfand-Shilov spaces  \cite[p.\ 113]{G-S} and \cite[Prop.\ 3.4]{Komatsu}.
\begin{lemma}\label{structural}
Let $(N_p)_{p \in \N}$ be a weight sequence satisfying $(\operatorname{lc})$ and $(\operatorname{dc})$. For every  $B \subset \mathcal{S}'_{(N_p)}(\R)$ equicontinuous there are $\nu \in \N$ and $C,h > 0$ such that for all $f \in B$ there are measurable functions $g_0, \ldots, g_\nu$ such that
$$
f = \sum_{n = 0}^\nu g_n^{(n)}  \mbox{ in $\mathcal{S}'_{(N_p)}(\R)$}
$$
and
$$
|g_n(x)| \leq C e^{N(h|x|)}, \qquad  x \in \R, 
$$
for all $n = 0, \ldots, \nu$.

\end{lemma}
Finally, we will also use the following well-known fact from distribution theory.
\begin{lemma}\label{deltas}
Let $\nu \in \N$. For all $N \geq \nu$, $c_0, \ldots, c_N \in \C$ and $g_0, \ldots, g_\nu \in  L^1_{\operatorname{loc}}(\R)$  it holds that 
$$
\sum_{n = 0}^N c_n \delta^{(n)} = \sum_{n = 0}^\nu g^{(n)}_n \mbox{ in $\mathcal{D}'(\R)$}
$$
implies that $c_\nu = \cdots = c_N = 0$.
\end{lemma}
\begin{proof}[Proof of Proposition \ref{Eidelheit-Borel}] We use Proposition \ref{Eidelheit-consequence} with  $E = \mathcal{S}_{(N_p)}(\R)$, $F = \mathcal{S}_{(N_p),0}(\R)$ and $(x'_n)_{n \in \N} = ((-1)^n\delta^{(n)})_{n \in \N}$. In view of  Lemma \ref{description-kernel}, condition $(1)$ follows from Lemma \ref{deltas} (with $\nu = 0$), while condition $(2)$ follows from Lemmas \ref{structural} and \ref{deltas}.  
\end{proof}

\section{The Stieltjes moment problem in $\mathcal{S}_\ast(0,\infty)$}\label{sect-main}
We are ready to characterize the surjectivity and the existence of a continuous linear right inverse of $\mathcal{M} : \mathcal{S}_{\ast}(0,\infty)  \rightarrow \Lambda_{\ast}$ (cf.\ Theorem \ref{main-theorem-intro}).

\begin{theorem}\label{main-theorem} Let $(M_p)_{p \in \N}$ be a weight sequence satisfying $(\operatorname{lc})$ and $(\operatorname{dc})$.
\begin{itemize}
\item[$(a)$] The following statements are equivalent:
\begin{itemize}
\item[$(i)$] $(M_p)_{p \in \N}$ satisfies $(\gamma_2)$.
\item[$(ii)$] $\mathcal{M}: \mathcal{S}_{(M_p)}(0,\infty) \rightarrow \Lambda_{(M_p)}$ has a continuous linear right inverse.
\item[$(iii)$] $\mathcal{M}: \mathcal{S}_{(M_p)}(0,\infty) \rightarrow \Lambda_{(M_p)}$ is surjective.
\end{itemize}
\item[$(b)$] The following statements are equivalent:
\begin{itemize}
\item[$(i)$] $(M_p)_{p \in \N}$ satisfies $(\gamma_2)$.
\item[$(ii)$] $\mathcal{M}: \Si_{\{M_p\}}(0,\infty) \rightarrow \Lambda_{\{M_p\}}$ is surjective.
\item[$(iii)$] $\mathcal{M}: \Sp_{\{M_p\}}(0,\infty) \rightarrow \Lambda_{\{M_p\}}$ is surjective.
\end{itemize}
\item[$(c)$] The following statements are equivalent:
\begin{itemize}
\item[$(i)$] $(M_p)_{p \in \N}$ satisfies $(\gamma_2)$ and $(\beta_2)$.
\item[$(ii)$] $\mathcal{M}: \Si_{\{M_p\}}(0,\infty) \rightarrow \Lambda_{\{M_p\}}$ has a continuous linear right inverse.
\item[$(iii)$] $\mathcal{M}: \Sp_{\{M_p\}}(0,\infty) \rightarrow \Lambda_{\{M_p\}}$ has a continuous linear right inverse.
\end{itemize}
\end{itemize}
\end{theorem}
In view of Lemma \ref{2-interpolating}, Theorem \ref{main-theorem} is a consequence of the following two results.

\begin{proposition}\label{reduction-1} Let $(M_p)_{p \in \N}$ be a weight sequence satisfying $(\operatorname{lc})$ and $(\operatorname{dc})$. Denote by $(N_p)_{p \in \N}$ its $2$-interpolating sequence.
\begin{itemize}
\item[$(a)$] $\mathcal{M}: \mathcal{S}_{(M_p)}(0,\infty) \rightarrow \Lambda_{(M_p)}$ is surjective (has a continuous linear right inverse) if and only if 
  $\mathcal{M}: \mathcal{S}_{(N_p)}^0(\R) \rightarrow \Lambda_{(N_p)}$  is surjective (has a continuous linear right inverse).
\item[$(b)$] If $\mathcal{M}: \Si_{\{N_p\}}^0(\R) \rightarrow \Lambda_{\{N_p\}}$ is surjective (has a continuous linear right inverse),
$\mathcal{M}: \Si_{\{M_p\}}(0,\infty) \rightarrow \Lambda_{\{M_p\}}$  is surjective (has a continuous linear right inverse) as well.
\item[$(c)$] If $\mathcal{M}: \Sp_{\{M_p\}}(0,\infty) \rightarrow \Lambda_{\{M_p\}}$  is surjective (has a continuous linear right inverse), 
$\mathcal{M}: \Sp_{\{N_p\}}^0(\R) \rightarrow \Lambda_{\{N_p\}}$  is surjective (has a continuous linear right inverse) as well.
\end{itemize}
\end{proposition}

\begin{proposition}\label{moment-flat}
Let $(N_p)_{p \in \N}$ be a weight sequence satisfying $(\operatorname{lc})$ and $(\operatorname{dc})$. 
\begin{itemize}
\item[$(a)$] The following statements are equivalent:
\begin{itemize}
\item[$(i)$] $(N_p)_{p \in \N}$ satisfies $(\gamma_1)$.
\item[$(ii)$] 
$\mathcal{M}: \mathcal{S}_{(N_p)}^0(\R) \rightarrow \Lambda_{(N_p)} $
has a continuous linear right inverse.
\item[$(iii)$] 
$\mathcal{M}: \mathcal{S}_{(N_p)}^0(\R) \rightarrow \Lambda_{(N_p)} $
is surjective.
\end{itemize}
\item[$(b)$] The following statements are equivalent:
\begin{itemize}
\item[$(i)$] $(N_p)_{p \in \N}$ satisfies $(\gamma_1)$.
\item[$(ii)$]
$\mathcal{M}: \Si_{\{N_p\}}^0(\R) \rightarrow \Lambda_{\{N_p\}} $
is surjective.
\item[$(iii)$] 
$\mathcal{M}: \Sp_{\{N_p\}}^0(\R) \rightarrow \Lambda_{\{N_p\}} $
is surjective.
\end{itemize}
\item[$(c)$] The following statements are equivalent:
\begin{itemize}
\item[$(i)$] $(N_p)_{p \in \N}$ satisfies $(\gamma_1)$ and $(\beta_2)$.
\item[$(ii)$] 
$\mathcal{M}: \Si_{\{N_p\}}^0(\R) \rightarrow \Lambda_{\{N_p\}} $
has a continuous linear right inverse.
\item[$(iii)$]   
$\mathcal{M}: \Sp^0_{\{N_p\}}(\R) \rightarrow \Lambda_{\{N_p\}} $
has a continuous linear right inverse.
\end{itemize}
\end{itemize}
\end{proposition}
The rest of this section is devoted to the proofs of the above two results.

\begin{proof}[Proof of Proposition \ref{reduction-1}] We only show that $\mathcal{M}: \mathcal{S}_{(M_p)}(0,\infty) \rightarrow \Lambda_{(M_p)}$ has a continuous linear right inverse if and only if 
  $\mathcal{M}: \mathcal{S}_{(N_p)}^0(\R) \rightarrow \Lambda_{(N_p)}$ does so; all other statements follow from a similar argument. We start with the direct implication. The proof is divided into two steps. 

STEP I: \emph{$\mathcal{M}_{e}: \mathcal{S}_{(N_p)}(0,\infty) \rightarrow \Lambda_{(M_p)}: \varphi \rightarrow (\mu_{2p}(\varphi))_{p \in \N}$ and $\mathcal{M}_{o}: \mathcal{S}_{(N_p)}(0,\infty) \rightarrow \Lambda_{(M_p)}: \varphi \rightarrow (\mu_{2p+1}(\varphi))_{p \in \N}$ have a continuous linear right inverse.} Let $R: \Lambda_{(M_p)} \rightarrow \mathcal{S}_{(M_p)}(0,\infty)$ be a continuous linear right inverse of $\mathcal{M}: \mathcal{S}_{(M_p)}(0,\infty) \rightarrow \Lambda_{(M_p)}$. Lemmas \ref{reduction-2-0}$(b)$ and \ref{reduction-2-1}$(b)$ imply that the mapping  
$$
T_{e}: \mathcal{S}_{(M_p)}(0,\infty) \rightarrow \mathcal{S}_{(N_p)}(0,\infty), \quad T_{e}(\varphi)(x)=
\left\{
	\begin{array}{ll}
		2x\varphi(x^2),  &  \mbox{$x > 0$}, \\ \\
		0,  &   \mbox{$x \leq 0$},
	\end{array}
\right. 
$$
is well-defined and continuous.  We claim that $T_{e} \circ R : \Lambda_{(M_p)} \rightarrow \mathcal{S}_{(N_p)}(0,\infty)$ is a continuous linear right inverse of  $\mathcal{M}_{e}: \mathcal{S}_{(N_p)}(0,\infty) \rightarrow \Lambda_{(M_p)}$. Let $a = (a_p)_{p \in \N} \in \Lambda_{(M_p)}$ be arbitrary. Then,
\begin{align*}
&\mu_{2p}((T_{e} \circ R)(a)) = \int_0^\infty x^{2p}[(T_{e} \circ R)(a)](x) \dx \\
&= 2\int_0^\infty  x^{2p}[R(a)](x^2) x\dx 
= \int_0^\infty  x^{p} [R(a)](x) \dx 
= \mu_{p}(R(a)) = a_p
\end{align*}
for all $p \in \N$. Likewise, Lemma  \ref{reduction-2-1}$(b)$ yields that the mapping 
$$
T_{o}: \mathcal{S}_{(M_p)}(0,\infty) \rightarrow \mathcal{S}_{(N_p)}(0,\infty), \quad T_{o}(\varphi)(x)=
\left\{
	\begin{array}{ll}
		2\varphi(x^2),  &  \mbox{$x > 0$}, \\ \\
		0,  &   \mbox{$x \leq 0$},
	\end{array}
\right. 
$$
is well-defined and continuous. We claim that $T_{o} \circ R : \Lambda_{(M_p)} \rightarrow \mathcal{S}_{(N_p)}(0,\infty)$ is a continuous linear right inverse of  $\mathcal{M}_{o}: \mathcal{S}_{(N_p)}(0,\infty) \rightarrow \Lambda_{(M_p)}$. Let $a = (a_p)_{p \in \N} \in \Lambda_{(M_p)}$ be arbitrary. Then,
\begin{align*}
&\mu_{2p+1}((T_{o} \circ R)(a)) = \int_0^\infty x^{2p+1}[(T_{o} \circ R)(a)](x)  \dx \\
&= 2\int_0^\infty  x^{2p}[R(a)](x^2) x \dx 
= \int_0^\infty  x^{p} [R(a)](x)\dx 
= \mu_{p}(R(a)) = a_p
\end{align*}
for all $p \in \N$.

STEP II: \emph{$\mathcal{M}: \mathcal{S}_{(N_p)}^0(\R) \rightarrow \Lambda_{(N_p)}$ has a continuous linear right inverse.}
Let $R_{e}: \Lambda_{(M_p)} \rightarrow \mathcal{S}_{(N_p)}(0,\infty)$ and $R_{o}: \Lambda_{(M_p)} \rightarrow \mathcal{S}_{(N_p)}(0,\infty)$ be  continuous linear right inverses of $\mathcal{M}_{e}: \mathcal{S}_{(N_p)}(0,\infty) \rightarrow \Lambda_{(M_p)}$  and $\mathcal{M}_{o}: \mathcal{S}_{(N_p)}(0,\infty) \rightarrow \Lambda_{(M_p)}$, respectively. Consider the continuous mappings
$$
T_{e} : \Lambda_{(N_p)} \rightarrow \Lambda_{(M_p)}: (a_p)_{p \in \N} \rightarrow (a_{2p})_{p \in \N} , \quad T_{o}: \Lambda_{(N_p)} \rightarrow \Lambda_{(M_p)}: (a_p)_{p \in \N}  \rightarrow (a_{2p+1})_{p \in \N},
$$
and 
\begin{gather*}
S_{e}:  \mathcal{S}_{(N_p)}(0,\infty) \rightarrow  \mathcal{S}_{(N_p)}^0(\R): \varphi \rightarrow \frac{\varphi + \varphi (- \cdot)}{2}, \\
 S_{o}:  \mathcal{S}_{(N_p)}(0,\infty) \rightarrow  \mathcal{S}_{(N_p)}^0(\R): \varphi \rightarrow \frac{\varphi - \varphi (- \cdot)}{2}.
\end{gather*}
Let $a = (a_p)_{p \in \N} \in  \Lambda_{(N_p)}$ be arbitrary. Then, 
\begin{gather*}
\mu_{2p}((S_{e} \circ R_{e} \circ T_{e})(a)) = a_{2p}, \qquad \mu_{2p+1}((S_{e} \circ R_{e} \circ T_{e})(a)) = 0, \\
\mu_{2p}((S_{o} \circ R_{o} \circ T_{o})(a)) = 0, \qquad \mu_{2p+1}((S_{o} \circ R_{o} \circ T_{o})(a)) = a_{2p+1},
\end{gather*}
for all $p \in \N$. Hence, $S_{e} \circ R_{e} \circ T_{e} +  S_{o} \circ R_{o} \circ T_{o}$ is a continuous linear right inverse of $\mathcal{M}: \mathcal{S}_{(N_p)}^0(\R) \rightarrow \Lambda_{(N_p)}$. 

Next, we  show the converse implication. Again, we divide the proof into two steps.

STEP I: \emph{$\mathcal{M}_{e}: \mathcal{S}_{(N_p)}(0,\infty) \rightarrow \Lambda_{(M_p)}: \varphi \rightarrow (\mu_{2p}(\varphi))_{p \in \N} $ has a continuous linear right inverse.}  Let $R: \Lambda_{(N_p)} \rightarrow \mathcal{S}_{(N_p)}^0(\R)$ be a continuous linear right inverse of $\mathcal{M}: \mathcal{S}_{(N_p)}^0(\R) \rightarrow \Lambda_{(N_p)}$. Consider the continuous mappings $T: \Lambda_{(M_p)} \rightarrow \Lambda_{(N_p)}$ given by $T((a_p)_{p \in \N}) = (b_p)_{p \in \N}$, where
$$
b_p = \left\{
	\begin{array}{ll}
		a_q, &  \mbox{$p = 2q$, $q \in \N$}, \\ \\
		0,  &   \mbox{otherwise},
	\end{array}
\right. 
$$
and 
$$
S:  \mathcal{S}^0_{(N_p)}(\R) \rightarrow  \mathcal{S}_{(N_p)}(0,\infty), \quad S(\varphi)(x) = \left\{
	\begin{array}{ll}
		\varphi(x) + \varphi(-x), &  \mbox{$x > 0$}, \\ \\
		0,  &   \mbox{$x \leq 0$}.
	\end{array}
\right. 
$$
We claim that $S \circ R \circ T: \Lambda_{(M_p)} \rightarrow \mathcal{S}_{(N_p)}(0,\infty)$ is a continuous linear right inverse of  $\mathcal{M}_{e}: \mathcal{S}_{(N_p)}(0,\infty) \rightarrow \Lambda_{(M_p)}$. Let $a = (a_p)_{p \in \N} \in \Lambda_{(M_p)}$  be arbitrary. Then,
\begin{align*}
&\mu_{2p}((S \circ R \circ T)(a)) = \int_0^\infty  x^{2p}[(S \circ R \circ T)(a)](x) \dx \\
&= \int_0^\infty x^{2p} [(R \circ T)(a)](x)  \dx + \int_0^\infty x^{2p}[(R \circ T)(a)](-x)  \dx \\
&=   \int_{-\infty}^\infty   x^{2p}[(R \circ T)(a)](x) \dx = \mu_{2p}(R(T(a))) = (T(a))_{2p} = a_p
\end{align*}
for all $p \in \N$.

STEP II: \emph{$\mathcal{M}: \mathcal{S}_{(M_p)}(0,\infty) \rightarrow \Lambda_{(M_p)}$ has a continuous linear right inverse.} Let $R_e: \Lambda_{(M_p)} \rightarrow \mathcal{S}_{(N_p)}(0,\infty)$ be a continuous linear right inverse of $\mathcal{M}_{\operatorname{e}}: \mathcal{S}_{(N_p)}(0,\infty) \rightarrow \Lambda_{(M_p)}$. Lemmas \ref{reduction-2-0}$(a)$  and \ref{reduction-2-1}$(a)$ imply that the mapping 
$$
T: \mathcal{S}_{(N_p)}(0,\infty) \rightarrow \mathcal{S}_{(M_p)}(0,\infty), \quad T(\varphi)(x)=
\left\{
	\begin{array}{ll}
		\frac{\varphi(x^{1/2})}{2x^{1/2}},  &  \mbox{$x > 0$}, \\ \\
		0,  &   \mbox{$x \leq 0$},
	\end{array}
\right. 
$$
is well-defined and continuous. We claim that $T \circ R_e : \Lambda_{(M_p)} \rightarrow \mathcal{S}_{(M_p)}(0,\infty)$ is a continuous linear right inverse of  $\mathcal{M}: \mathcal{S}_{(M_p)}(0,\infty) \rightarrow \Lambda_{(M_p)}$. Let $a = (a_p)_{p \in \N} \in \Lambda_{(M_p)}$ be arbitrary. Then,
\begin{align*}
&\mu_{p}((T \circ R_e)(a)) = \int_0^\infty x^{p}[(T \circ R_e)(a)](x)  \dx = \int_0^\infty x^{p} \frac{[R_e(a)](x^{1/2})}{2x^{1/2}}  \dx \\
&= \int_0^\infty x^{2p} [R_e(a)](x)  \dx = \mu_{2p}(R_e(a)) = a_p
\end{align*}
for all $p \in \N$.
\end{proof}

\begin{proof}[Proof of Proposition \ref{moment-flat}]
$(a)$ $(i) \Rightarrow (ii)$  We apply Proposition \ref{functional-analysis} with $E = \Sb_{(N_p)}(\R)$, $F = \Lambda_{(N_p)}$, $G = \C^\N$, $T = \mathcal{M}:  \Sb_{(N_p)}(\R) \rightarrow \Lambda_{(N_p)}$ and $S = \mathcal{B}:  \Sb_{(N_p)}(\R) \rightarrow \C^\N$. 
Lemma \ref{ind-proj-description} yields that the topology induced by $E = \Sb_{(N_p)}(\R)$ on $\ker S = \Sb^0_{(N_p)}(\R)$ coincides with the original topology of $\Sb^0_{(N_p)}(\R)$.
We now verify conditions $(1)$-$(4)$: $(1)$ Obvious. $(2)$ The sequence of standard unit vectors is an absolute Schauder basis in $\Lambda_{(N_p)}$. $(3)$ This has been shown in Proposition \ref{Eidelheit-Borel-1}. $(4)$ Set $E_0 =\mathcal{F}^{-1}\left(\mathcal{D}^{(N_p)}_{[-1,1]} \right)$ and endow it with the topology generated by the system of seminorms $\left \{ p \circ \mathcal{F} \, | \, p \in \csn(\mathcal{D}^{(N_p)}_{[-1,1]}) \right\}$. Notice that $\mathcal{D}^{(N_p)}_{[-1,1]} \subset \mathcal{S}^{(N_p)}(\R)$ with continuous inclusion mapping and recall that $\mathcal{F}: \Sb_{(N_p)}(\R) \rightarrow \Sb^{(N_p)}(\R)$ is a topological isomorphism. Hence, $(4.1)$ is clear, while $(4.2)$ follows from Theorem \ref{Borel-P}$(a)$. Finally, 
$$
|\varphi^{(p)}(0)| \leq \frac{1}{2\pi} \int_{-1}^{1} |x|^p |\widehat{\varphi}(x)|  \dx \leq \frac{1}{\pi} \| \widehat{\varphi}\|_{L^\infty([-1,1])}, \qquad p \in \N,
$$
for all $\varphi \in E_0$, whence $(4.3)$ holds.

$(ii) \Rightarrow (iii)$ Trivial.

$(iii) \Rightarrow (i)$ In particular, $\mathcal{M}: \mathcal{S}_{(N_p)}(\R) \rightarrow \Lambda_{(N_p)}$ is surjective. By taking the Fourier transform, we obtain that $\mathcal{B}: \mathcal{S}^{(N_p)}(\R) \rightarrow \Lambda_{(N_p)}$ is surjective. Choose $\varphi \in \mathcal{S}^{(N_p)}(\R)$ such that $\varphi^{(p)}(0) = \delta_{0,p}$ for all $p \in \N$. Set $\psi = \varphi - 1$. Then,
$$
\sup_{p \in \N} \sup_{x \in \R} \frac{h^p|\psi^{(p)}(x)|}{N_p} < \infty
$$ 
for all $h > 0$ and $\psi^{(p)}(0) = 0$ for all $p \in \N$. Since $\lim_{|x|\to \infty} \varphi(x) = 0$, $\psi$ is not identically zero. Hence, the Denjoy-Carleman theorem implies that $(N_p)_{p \in \N}$ satisfies $(\gamma)$.  By Theorem \ref{Borel-P}$(a)$, it therefore suffices to show that the mapping $\mathcal{B}: \mathcal{D}^{(N_p)}_{[-1,1]} \rightarrow  \Lambda_{(N_p)}$ is surjective. Let $a \in \Lambda_{(N_p)}$ be arbitrary and choose $\varphi \in \mathcal{S}^{(N_p)}(\R)$ such that $\mathcal{B}(\varphi) = a$. Pick $\psi \in \mathcal{D}^{(N_p)}_{[-1,1]}$ such that $\psi \equiv 1$ in a neighbourhood of 0. Then, $\varphi\psi \in  \mathcal{D}^{(N_p)}_{[-1,1]}$  and $\mathcal{B}(\varphi \psi) = a$.

 $(b)$ $(i) \Rightarrow (ii)$ We apply Lemma \ref{surjectivity-kernel} with $E = \Si_{\{N_p\}}(\R)$, $F = \Lambda_{\{N_p\}}$, $G = \C^\N$, $T = \mathcal{M}:  \Si_{\{N_p\}}(\R) \rightarrow \Lambda_{\{N_p\}}$ and $S = \mathcal{B}:  \Si_{\{N_p\}}(\R) \rightarrow \C^\N$.  Theorem \ref{Borel-P}$(b)$ and  the inclusion $\mathcal{D}^{\{N_p\}}_{[-1,1]} \subset  \Si^{\{N_p\}}(\R)$ yield that $\mathcal{B}:  \Si^{\{N_p\}}(\R) \rightarrow \Lambda_{\{N_p\}}$ is surjective. By taking the Fourier transform, we obtain that $T$ is surjective.  $S_{|\ker T}$ is surjective because of Proposition \ref{Eidelheit-Borel} and the inclusion $\mathcal{S}_{(N_p)}(\R) \subset  \Si_{\{N_p\}}(\R)$.

 $(ii) \Rightarrow (iii)$ This follows from the continuous inclusion  $\Si_{\{N_p\}}(\R) \subset   \Sp_{\{N_p\}}(\R)$.

 $(iii) \Rightarrow (i)$ This can be shown in a similar way as $(iii) \Rightarrow (i)$ from part $(a)$.

$(c)$ $(i) \Rightarrow (ii)$ We apply Proposition \ref{functional-analysis} with $E = \Si_{\{N_p\}}(\R)$, $F = \Lambda_{\{N_p\}}$, $G = \C^\N$, $T = \mathcal{M}:  \Si_{\{N_p\}}(\R) \rightarrow \Lambda_{\{N_p\}}$ and $S = \mathcal{B}:  \Si_{\{N_p\}}(\R) \rightarrow \C^\N$. 
Lemma \ref{ind-proj-description} yields that the topology induced by $E =\Si_{\{N_p\}}(\R)$ on $\ker S = \Si^0_{\{N_p\}}(\R)$ coincides with the original topology of $\Si^0_{\{N_p\}}(\R)$.
We now verify conditions $(1)$-$(4)$: $(1)$ $\Si_{\{N_p\}}(\R)$ is complete by Proposition \ref{complete}. $(2)$ The sequence of standard unit vectors is an absolute Schauder basis in $\Lambda_{\{N_p\}}$. $(3)$ This follows from Proposition \ref{Eidelheit-Borel-1} and the continuous inclusion  $\mathcal{S}_{(N_p)}(\R) \subset  \Si_{\{N_p\}}(\R)$. $(4)$ Set $E_0 =\mathcal{F}^{-1}\left(\mathcal{D}^{\{N_p\}}_{[-1,1]} \right)$ and endow it with the topology generated by the system of seminorms $\left \{ p \circ \mathcal{F} \, | \, p \in \csn(\mathcal{D}^{\{N_p\}}_{[-1,1]}) \right\}$. Notice that $\mathcal{D}^{\{N_p\}}_{[-1,1]} \subset \Si^{\{N_p\}}(\R)$ with continuous inclusion mapping and recall that $\mathcal{F}: \Si_{\{N_p\}}(\R) \rightarrow \Si^{\{N_p\}}(\R)$ is a topological isomorphism. Hence, $(4.1)$ is clear, while $(4.2)$ follows from Theorem \ref{Borel-P}$(c)$. Finally, 
$$
|\varphi^{(p)}(0)| \leq \frac{1}{2\pi} \int_{-1}^{1} |x|^p |\widehat{\varphi}(x)|  \dx \leq \frac{1}{\pi} \| \widehat{\varphi}\|_{L^\infty([-1,1])}, \qquad p \in \N,
$$
for all $\varphi \in E_0$, whence $(4.3)$ holds.


$(ii) \Rightarrow (iii)$  This follows from the inclusion  $\Si_{\{N_p\}}(\R) \subset   \Sp_{\{N_p\}}(\R)$.

$(iii) \Rightarrow (i)$  This can be shown in a similar way as $(iii) \Rightarrow (i)$ from part $(a)$.
\end{proof}
\section{The Stieltjes moment problem in $\mathcal{S}^\dagger_\ast(0,\infty)$ and the Borel-Ritt problem in spaces of ultraholomorphic functions  on the upper half-plane}\label{sect-cor}
In this final section, we show an analogue of Theorem \ref{main-theorem-1} both for the Stieltjes moment problem in Gelfand-Shilov spaces of type $\mathcal{S}^\dagger_\ast(0,\infty)$ and the Borel-Ritt problem in spaces of ultraholomorphic functions  on the upper half-plane  $\HH = \{ z \in \C \, | \,  \Im m z > 0 \}$.

We start by introducing Gelfand-Shilov spaces of type $\mathcal{S}^\dagger_\ast$. Let $(M_p)_{p \in \N}$ and $(A_p)_{p \in \N}$ be two weight sequences. For $h> 0$ we write $\mathcal{S}^{A_p,h}_{M_p,h}(\R)$ for the Banach space consisting of all $\varphi \in C^{\infty}(\R)$ such that
$$
\| \varphi\|_{\mathcal{S}^{A_p,h}_{M_p,h}} := \sup_{p,q \in \N} \sup_{x \in \R} \frac{h^{p+q} |x^p\varphi^{(q)}(x)|}{A_qM_p} < \infty.
$$
We set
$$
\mathcal{S}^{(A_p)}_{(M_p)}(\R) := \varprojlim_{h \to \infty} \mathcal{S}^{A_p,h}_{M_p,h}(\R), \qquad \mathcal{S}^{\{A_p\}}_{\{M_p\}}(\R) := \varinjlim_{h \to 0^+} \mathcal{S}^{A_p,h}_{M_p,h}(\R).
$$
$\mathcal{S}^{(A_p)}_{(M_p)}(\R)$ is a Fr\'echet space, while $\mathcal{S}^{\{A_p\}}_{\{M_p\}}(\R)$ is an $(LB)$-space.  
Similarly as before, we will sometimes use $\mathcal{S}^{\dagger}_{\ast}(\R)$ as a common notation for $\mathcal{S}^{(A_p)}_{(M_p)}(\R)$ and $\mathcal{S}^{\{A_p\}}_{\{M_p\}}(\R)$. If both $(M_p)_{p \in \N}$ and $(A_p)_{p \in \N}$ satisfy $(\operatorname{lc})$, $(\operatorname{dc})$ and $(\gamma)$,  the Fourier transform is a topological isomorphism from $\mathcal{S}^{\dagger}_{\ast}(\R)$ onto $\mathcal{S}_{\dagger}^{\ast}(\R)$ (cf.\ \cite[Sect.\ IV.6]{G-S} and  \cite[Lemma 4.1]{Komatsu}).
 
Let $h > 0$. We define the following closed subspace of $\mathcal{S}^{A_p,h}_{M_p,h}(\R)$
$$
\mathcal{S}^{A_p,h}_{M_p,h}(0,\infty) := \{ \varphi \in \mathcal{S}^{A_p,h}_{M_p,h}(\R) \, | \, \operatorname{supp} \varphi \subseteq  [0, \infty) \}
$$
and endow it with the norm $\| \, \cdot \, \|_{\mathcal{S}^{A_p,h}_{M_p,h}}$.  Hence, it becomes a Banach space. We set
$$
\mathcal{S}^{(A_p)}_{(M_p)}(0,\infty) := \varprojlim_{h \to \infty} \mathcal{S}^{A_p,h}_{M_p,h}(0,\infty), \qquad \mathcal{S}^{\{A_p\}}_{\{M_p\}}(0,\infty) := \varinjlim_{h \to 0^+} \mathcal{S}^{A_p,h}_{M_p,h}(0,\infty).
$$
$\mathcal{S}^{(A_p)}_{(M_p)}(0,\infty)$ is a Fr\'echet space, while $\mathcal{S}^{\{A_p\}}_{\{M_p\}}(0,\infty)$ is an $(LB)$-space.
Notice that 
\begin{equation}
\mathcal{S}^{\dagger}_{\ast}(0,\infty) = \{ \varphi \in \mathcal{S}^{\dagger}_{\ast}(\R)  \, | \, \operatorname{supp} \varphi \subseteq  [0, \infty) \}
\label{DFS}
\end{equation}
as sets. In the Beurling case, it is clear that \eqref{DFS} also holds topologically if we endow the space at the right-hand side with the relative topology induced by $\mathcal{S}^{(A_p)}_{(M_p)}(\R)$. If $(M_p)_{p \in \N}$ satisfies $\operatorname{(lc)}$ and $\operatorname{(dc)}$, the corresponding statement also holds in the Roumieu case. By \cite[Prop.\ 3.4]{Komatsu}, these assumptions imply that $\mathcal{S}^{\{A_p\}}_{\{M_p\}}(\R)$ is a $(DFS)$-space, whence the result follows from the fact that a closed subspace of a $(DFS)$-space is again a $(DFS)$-space and De Wilde's open mapping theorem. 

The image of $\mathcal{S}^{\dagger}_{\ast}(0,\infty)$ under the Fourier transform can be described as follows.
 
\begin{lemma}\label{Fourier-char-supp} \emph{(cf.\ \cite[Prop.\ 2.1]{C-C-K})} Let $(M_p)_{p \in \N}$ and $(A_p)_{p \in \N}$ be two weight sequences satisfying $(\operatorname{lc})$, $(\operatorname{dc})$  and $(\gamma)$. Let $\psi  \in \mathcal{S}_{\dagger}^{\ast}(\R)$. Then, $\psi \in \mathcal{F}(\mathcal{S}^{\dagger}_{\ast}(0,\infty))$ if and only if there is $\Psi: \overline{\HH} \rightarrow \C$ satisfying the following conditions:
\begin{itemize}
\item[$(i)$] $\Psi_{|\R} = \psi$.
\item[$(ii)$] $\Psi$ is continuous on $\overline{\HH}$ and holomorphic on $\HH$.
\item[$(iii)$] $\lim_{z \in \overline{\HH}, z \to \infty} \Psi(z) = 0$.
\end{itemize}
\end{lemma}
Next, we define spaces of ultraholomorphic functions on $\HH$. Given an open subset $\Omega \subseteq \C$, we denote by $\mathcal{O}(\Omega)$ the space of holomorphic functions on $\Omega$. Let $(M_p)_{p \in \N}$ be a weight sequence. For $h>0$ we write $\mathcal{A}^{M_p,h}(\HH)$ for the Banach space consisting of all $f \in \mathcal{O}(\HH)$ such that
$$
\| f \|_{\mathcal{A}^{M_p,h}} := \sup_{p \in \N} \sup_{z \in \HH} \frac{h^p|f^{(p)}(z)|}{M_p} < \infty.
$$
We set 
$$
\mathcal{A}^{(M_p)}(\HH) := \varprojlim_{h \to \infty} \mathcal{A}^{M_p,h}(\HH), \qquad \mathcal{A}^{\{M_p\}}(\HH) := \varinjlim_{h \to 0^+} \mathcal{A}^{M_p,h}(\HH).
$$
$\mathcal{A}^{(M_p)}(\HH)$ is a Fr\'echet space, while $\mathcal{A}^{\{M_p\}}(\HH)$ is an $(LB)$-space. Let $f \in \mathcal{A}^*(\HH)$ be arbitrary. Since $f$ and all its derivatives are Lipschitz on $\HH$, it holds that
$$
f_p(x) = \lim_{z \to x, z \in \HH} f^{(p)}(z) \in \C, \qquad x \in \R,
$$
exists for all $p \in \N$. Moreover, $f_0 \in C^\infty(\R)$ and $f_0^{(p)} = f_p$ for all $p \in \N$. From now on, we simply write $f_0 = f$.  The \emph{asymptotic Borel mapping}
$$
\mathcal{B}: \mathcal{A}^{\ast}(\HH) \rightarrow \Lambda_{\ast}: f \rightarrow (f^{(p)}(0))_{p \in \N}
$$
is well-defined and continuous.

We are ready to prove the two main results of this section.
\begin{theorem}\label{main-theorem-1} Let $(M_p)_{p \in \N}$ be a weight sequence satisfying $(\operatorname{lc})$ and $(\operatorname{dc})$, and let $(A_p)_{p \in \N}$ be a weight sequence satisfying $(\operatorname{lc})$ and $(\gamma)$.
\begin{itemize}
\item[$(a)$] The following statements are equivalent:
\begin{itemize}
\item[$(i)$] $(M_p)_{p \in \N}$ satisfies $(\gamma_2)$.
\item[$(ii)$] $\mathcal{M}: \mathcal{S}^{(A_p)}_{(M_p)}(0,\infty) \rightarrow \Lambda_{(M_p)}$ has a continuous linear right inverse.
\item[$(iii)$] $\mathcal{M}: \mathcal{S}^{(A_p)}_{(M_p)}(0,\infty) \rightarrow \Lambda_{(M_p)}$ is surjective.
\end{itemize}
\item[$(b)$] $(M_p)_{p \in \N}$ satisfies $(\gamma_2)$ if and only if  $\mathcal{M}: \mathcal{S}^{\{A_p\}}_{\{M_p\}}(0,\infty) \rightarrow \Lambda_{\{M_p\}}$ is surjective.
\item[$(c)$] $(M_p)_{p \in \N}$ satisfies $(\gamma_2)$ and $(\beta_2)$ if and only if  $\mathcal{M}: \mathcal{S}^{\{A_p\}}_{\{M_p\}}(0,\infty) \rightarrow \Lambda_{\{M_p\}}$ has a continuous linear right inverse.
\end{itemize}
\end{theorem}
\begin{remark}
In \cite[Thm.\ 3.5]{D-J-S}, the direct implication of Theorem \ref{main-theorem-1}$(b)$ was shown under the assumptions $\operatorname{(slc)}$ (= $(M_p/p!)_{p \in \N}$ satisfies $\operatorname{(lc)}$)  and $\operatorname{(mg)}$, while the converse implication was shown under the assumptions  $\operatorname{(slc)}$ and $\operatorname{(dc)}$.
\end{remark}
\begin{theorem}\label{main-theorem-2} Let $(M_p)_{p \in \N}$ be a weight sequence satisfying $(\operatorname{lc})$ and $(\operatorname{dc})$.
\begin{itemize}
\item[$(a)$] The following statements are equivalent:
\begin{itemize}
\item[$(i)$] $M_p$ satisfies $(\gamma_2)$.
\item[$(ii)$] $\mathcal{B} : \mathcal{A}^{(M_p)}(\HH) \rightarrow \Lambda_{(M_p)}$ has a continuous linear right inverse.
\item[$(iii)$] $\mathcal{B} : \mathcal{A}^{(M_p)}(\HH) \rightarrow \Lambda_{(M_p)}$  is surjective.
\end{itemize}
\item[$(b)$] $(M_p)_{p \in \N}$ satisfies $(\gamma_2)$ if and only if $\mathcal{B} : \mathcal{A}^{\{M_p\}}(\HH) \rightarrow \Lambda_{\{M_p\}}$ is surjective.
\item[$(c)$] $(M_p)_{p \in \N}$ satisfies $(\gamma_2)$ and $(\beta_2)$  if and only if $\mathcal{B} : \mathcal{A}^{\{M_p\}}(\HH) \rightarrow \Lambda_{\{M_p\}}$ has a continuous linear right inverse.
\end{itemize}
\end{theorem}
\begin{remark}\label{details}
Theorem \ref{main-theorem-2} improves various results from \cite{S-V, Thilliez, JG-S-S} in the special case of the upper half-plane:
The implication $(i) \Rightarrow (iii)$ from Theorem \ref{main-theorem-1}$(a)$ was shown in \cite[Cor.\ 3.4.1]{Thilliez} under the assumptions $\operatorname{(slc)}$ and $\operatorname{(mg)}$; the existence of a continuous linear right inverse of $\mathcal{B} : \mathcal{A}^{(M_p)}(\HH) \rightarrow \Lambda_{(M_p)}$ was shown in \cite[Thm.\ 4.4 and Thm.\ 4.5]{S-V} under the assumptions $\operatorname{(lc)}$ and $(\gamma_3)$; the direct implication of Theorem \ref{main-theorem-2}$(b)$ was shown in \cite[Thm.\ 3.2.1]{Thilliez} under the assumptions $\operatorname{(slc)}$ and $\operatorname{(mg)}$, while the converse implication was shown in \cite[Thm.\ 4.14]{JG-S-S} under the assumptions $\operatorname{(slc)}$ and $\operatorname{(dc)}$;  the existence of a continuous linear right inverse of $\mathcal{B} : \mathcal{A}^{\{M_p\}}(\HH) \rightarrow \Lambda_{\{M_p\}}$ was shown in \cite[Thm.\ 5.4 and Thm.\ 5.5]{S-V} under the assumptions $\operatorname{(lc)}$, $(\gamma_3)$ and $(\beta_2)$.
\end{remark}
In view of Theorem \ref{main-theorem}, Theorems \ref{main-theorem-1} and \ref{main-theorem-2} are both consequences of the following result; it is essentially shown in \cite[Sect.\ 3]{D-J-S}, but we repeat the argument here for the sake of completeness.

\begin{proposition} \label{equivalences} Let $(M_p)_{p \in \N}$ be a weight sequence satisfying $(\operatorname{lc})$ and $(\operatorname{dc})$, and let $(A_p)_{p \in \N}$ be a weight sequence satisfying $(\operatorname{lc})$ and $(\gamma)$. 
\begin{itemize}
\item[$(a)$] The following statements are equivalent:
\begin{itemize}
\item[$(i)$] $\mathcal{M} : \mathcal{S}^{(A_p)}_{(M_p)}(0,\infty) \rightarrow \Lambda_{(M_p)}$ is surjective (has a continuous linear right inverse).
\item[$(ii)$] $\mathcal{M} : \mathcal{S}_{(M_p)}(0,\infty) \rightarrow \Lambda_{(M_p)}$ is surjective (has a continuous linear right inverse).
\item[$(iii)$] $\mathcal{B} : \mathcal{A}^{(M_p)}(\HH) \rightarrow \Lambda_{(M_p)}$ is surjective (has a continuous linear right inverse).
\end{itemize}
\item[$(b)$] The following statements are equivalent:
\begin{itemize}
\item[$(i)$] $\mathcal{M} : \mathcal{S}^{\{A_p\}}_{\{M_p\}}(0,\infty) \rightarrow \Lambda_{\{M_p\}}$ is surjective (has a continuous linear right inverse).
\item[$(ii)$] $\mathcal{M} : \Si_{\{M_p\}}(0,\infty) \rightarrow \Lambda_{\{M_p\}}$ is surjective (has a continuous linear right inverse).
\item[$(iii)$] $\mathcal{B} : \mathcal{A}^{\{M_p\}}(\HH) \rightarrow \Lambda_{\{M_p\}}$ is surjective (has a continuous linear right inverse).
\end{itemize}
\end{itemize}
\end{proposition}
\begin{proof}
We only show the equivalences about the existence of a continuous linear right inverse stated in $(a)$; all other cases follow from a similar argument. 

$(i) \Rightarrow (ii)$ This follows from the continuous inclusion $\mathcal{S}^{(A_p)}_{(M_p)}(0,\infty) \subset \mathcal{S}_{(M_p)}(0,\infty)$.

$(ii) \Rightarrow (iii)$ We define the Laplace transform of an element $\varphi \in  \mathcal{S}_{(M_p)}(0,\infty)$ as
$$
\mathcal{L}(\varphi)(z) := \int_{0}^\infty \varphi(t) e^{itz} \dt, \qquad z \in \HH.
$$
$\mathcal{L}:  \mathcal{S}_{(M_p)}(0,\infty) \rightarrow \mathcal{A}^{(M_p)}(\HH)$ is well-defined and continuous, and $\mathcal{L}(\varphi)^{(p)}(0) = i^p \mu_p(\varphi)$ for all $p \in \N$. Let $R: \Lambda_{(M_p)} \rightarrow  \mathcal{S}_{(M_p)}(0,\infty)$ be a continuous linear right inverse of $\mathcal{M} : \mathcal{S}_{(M_p)}(0,\infty) \rightarrow \Lambda_{(M_p)}$. Consider the continuous mapping 
$$
T: \Lambda_{(M_p)}  \rightarrow \Lambda_{(M_p)}, \quad T((a_p)_{p \in \N}) = ((-i)^pa_p)_{p \in \N}.
$$ 
Then, $\mathcal{L} \circ R \circ T: \Lambda_{(M_p)} \rightarrow \mathcal{A}^{(M_p)}(\HH)$ is a continuous linear right inverse of $\mathcal{B} : \mathcal{A}^{(M_p)}(\HH) \rightarrow \Lambda_{(M_p)}$.

$(iii) \Rightarrow (i)$ We start by making some preliminary observations about the weight sequences $(M_p)_{p \in \N}$ and $(A_p)_{p \in \N}$. By  \cite[Lemma 2.2]{D-J-S}, we may assume without loss of generality that $(A_p)_{p \in \N}$ satisfies $(\operatorname{dc})$. Next,  choose $f \in \mathcal{A}^{(M_p)}(\HH)$ such that $f^{(p)}(0) = \delta_{1,p}$ for all $p \in \N$. Set $\varphi(x) = f(x) - x$ for $x \in \R$. Then,
$$
\sup_{p \in \N} \sup_{x \in [-R,R]} \frac{h^p|\varphi^{(p)}(x)|}{M_p} < \infty
$$ 
for all $h > 0$ and $R > 0$, and $\varphi^{(p)}(0) = 0$ for all $p \in \N$. Since $f$ is bounded, $\varphi$ is not identically zero. Hence, the Denjoy-Carleman theorem implies that $(M_p)_{p \in \N}$ satisfies $(\gamma)$. Consequently, we have that  $\lim_{p \to \infty} (p!/M_p)^{1/p} = 0$ \cite[Lemma 4.1]{Komatsu}. We now turn to the actual proof. It is based on the following observation \cite[Lemma 3.6]{D-J-S}: Let $(a_p)_{p \in \N} \in \C^\N$ and let $G \in C^\infty((-\delta,\delta))$, $\delta > 0$, with $G(0) \neq 0$. Set
$$
b_p = \sum_{n = 0}^p \binom{p}{n} a_n \left ( \frac{1}{G} \right)^{(p-n)} (0), \qquad p \in \N.
$$
Then,
$$
\sum_{n=0}^p \binom{p}{n} b_n G^{(p-n)}(0) = a_p, \qquad p\in \N.
$$
Set $V= \{ z \in \C \, | \, \Im m z > -1 \}$. By  \cite[Lemma 3.1]{D-J-S} and \cite[Lemma 4.3]{Komatsu}, there is $G \in \mathcal{O}(V)$ satisfying the following properties: 
\begin{itemize}
\item[$(i)$] $G$ does not vanish on $V$.
\item[$(ii)$] $\displaystyle \sup_{z \in V} |G(z)|e^{{A}(h|z|)} < \infty$ for all $h > 0$.
\item[$(iii)$] $\displaystyle \sup_{p \in \N}\sup_{x \in \R} \frac{|G^{(p)}(x)|e^{A(h|x|)}}{2^pp!} < \infty$ for all $h > 0$.
\end{itemize}
Lemma \ref{Fourier-char-supp} implies that the mapping $T: \mathcal{A}^{(M_p)}(\HH) \rightarrow \mathcal{S}^{(A_p)}_{(M_p)}(0,\infty): f \rightarrow  \mathcal{F}^{-1}((fG)_{|\R})$ is well-defined. Since $\mathcal{F}:  \mathcal{S}^{(A_p)}_{(M_p)}(\R) \rightarrow  \mathcal{S}_{(A_p)}^{(M_p)}(\R)$ is a topological isomorphism and \eqref{DFS} holds topologically, $T$ is also continuous.  Next, the Cauchy estimates yield that 
$$
\sup_{p \in \N} \frac{|(1/G)^{(p)}(0)|}{2^p p!}< \infty.
$$ 
Hence, the mapping
$$
S: \Lambda_{(M_p)} \rightarrow \Lambda_{(M_p)}, \quad S(a) = \left( \sum_{n = 0}^p \binom{p}{n} i^na_n \left ( \frac{1}{G} \right)^{(p-n)} (0) \right)_{p \in \N},
$$
is well-defined and continuous.  Let $R: \Lambda_{(M_p)} \rightarrow  \mathcal{A}^{(M_p)}(\HH)$ be a continuous linear right inverse of $\mathcal{B} : \mathcal{A}^{(M_p)}(\HH) \rightarrow \Lambda_{(M_p)}$. Then, $T \circ R \circ S : \Lambda_{(M_p)} \rightarrow \mathcal{S}^{(A_p)}_{(M_p)}(0,\infty)$ is a continuous linear right inverse of $\mathcal{M} : \mathcal{S}^{(A_p)}_{(M_p)}(0,\infty) \rightarrow \Lambda_{(M_p)}$.
\end{proof}

\end{document}